\documentclass[article]{amsart}

\usepackage{amssymb,amsfonts,amsmath,amsthm}
\usepackage[all,arc]{xy}
\usepackage{enumerate}
\usepackage{mathrsfs}
\usepackage[toc,page]{appendix}
\usepackage{graphicx}
\usepackage[font=small,labelfont=bf]{caption}
\usepackage[left=3cm, right=3cm, bottom=3cm]{geometry}
\usepackage{tabularx}
\usepackage{url}
\usepackage{color}
\usepackage{tikz-cd}
\usepackage{hyperref} 

\newtheorem{thm}{Theorem}[section]
\newtheorem{cor}[thm]{Corollary}
\newtheorem{prop}[thm]{Proposition}

\newtheorem{conj}[thm]{Conjecture}

\theoremstyle{definition}
\newtheorem{defn}[thm]{Definition}

\theoremstyle{remark}
\newtheorem{rem}[thm]{Remark}

\makeatletter
\let\c@equation\c@thm
\makeatother
\numberwithin{equation}{section}

\newcommand{\pardeg}{\mathrm{pardeg}}

\title{Poisson Structures on Moduli Spaces of Higgs Bundles over Stacky Curves}
\author{Georgios Kydonakis, Hao Sun and Lutian Zhao}
\begin{document}

\begin{abstract}
We demonstrate the construction of Poisson structures via Lie algebroids on moduli spaces of twisted stable Higgs bundles over stacky curves.  The construction provides new examples of Poisson structures on such moduli spaces. Special attention is paid at moduli spaces of parabolic Higgs bundles over a root stack.
\end{abstract}
\maketitle

\renewcommand{\thefootnote}{\fnsymbol{footnote}}

\section{Introduction}

Symplectic or Poisson structures on moduli spaces over a complex algebraic curve $X$ have been obtained in a variety of cases. A basic paradigm involves the $G$-character varieties, as moduli spaces of fundamental group representations into a connected complex reductive group $G$. These are finite-dimensional complex symplectic manifolds when $X$ is compact and the symplectic structure in this case was first conceived analytically using the method of symplectic reduction from infinite dimensional spaces in the seminal work of Atiyah and Bott \cite{AtBo}. At the same time, a more topological approach to the natural symplectic structure on such spaces of fundamental group representations was proposed by Goldman \cite{Go} interpreting these structures in terms of the intersection pairing on the underlying topological surface. The method of symplectic reduction was next further developed by Hitchin in \cite{Hit87} to produce K\"{a}hler and hyperk\"{a}hler structures on the moduli space of stable Higgs bundles via the non-abelian Hodge correspondence and their counterparts to moduli spaces of solutions to the self-duality equations.

In the case when the curve is noncompact, the symplectic structure generalizes to a Poisson structure. The noncompact case is actually equivalent to equipping $X$ with a reduced effective divisor $D$. The symplectic leaves consist of equivalence classes of connections with fixed conjugacy class of local holonomy around each point of a fixed reduced effective divisor on $X$. A primary description of this theory appeared in the book of Atiyah \cite{Atbook}, while in the article of Audin \cite{Audin} a review of several approaches is presented. In the particular situation of parabolic vector bundles on $X$ with trivial flags in the parabolic structure, Poisson structures on moduli spaces of twisted Higgs bundles were obtained independently by Bottacin \cite{Bot} and Markman \cite{Mar}. 

Note that from the point of view of the tame nonabelian Hodge correspondence as described by Simpson \cite{Simp-noncompact}, parabolic Higgs bundles correspond to filtered local systems which are regarded as representations of the fundamental group $\pi_1(X\backslash D)$. In the wild case, a nonabelian Hodge correspondence when the structure group is $\text{GL}(n, \mathbb{C})$ was obtained collectively from the works of Biquard-Boalch \cite{BiBo} and Sabbah \cite{Sa}. For a general connected complex reductive group $G$, Boalch introduced wild character varieties to classify meromorphic $G$-connections with higher order poles, and constructed Poisson structures on them (see for instance \cite{Bo3,Bo4}, and the survey article \cite{BoSurvey}).

In this article, we exhibit a wider class of Poisson structures on moduli spaces of Higgs bundles over stacky curves, demonstrating an intrinsic property that justifies the existence of the Poisson structure. The principal method by which we shall obtain Poisson structures on the moduli spaces of interest is via the duals of Lie algebroids. 

In \cite{LoMa}, Logares and Martens considered moduli spaces $\mathsf{\mathcal{P}}_{\alpha}$ of $\alpha$-semistable parabolic Higgs bundles over an algebraic curve. In fact, the open subset $\mathsf{\mathcal{P}}_{\alpha }^{0}\subset {{\mathsf{\mathcal{P}}}_{\alpha }}$ of pairs involving a stable underlying parabolic bundle is a vector bundle over the moduli space  ${{\mathsf{\mathcal{N}}}_{\alpha }}$ of stable parabolic bundles. Logares and Martens showed that the dual of this vector bundle is an Atiyah algebroid associated to a principal bundle over the space ${{\mathsf{\mathcal{N}}}_{\alpha }}$, thus admitting a Poisson structure. On the other hand, there exists a bi-vector field on $\mathcal{P}_{\alpha}$ which agrees with the Poisson bracket on $\mathcal{P}^{0}_{\alpha}$, thus establishing the Poisson structure on the entire  $\mathcal{P}_{\alpha}$. It is important here the fact that the Atiyah sequence of the algebroid naturally follows from the deformation theory of parabolic vector bundles and a Serre duality map in hypercohomology plays a pivotal role in the definition of the Poisson bracket. These ideas provided important motivation for the development of the present work.

Let $\mathcal{X}$ be a smooth projective Deligne-Mumford stack over $\mathbb{C}$ and let $X$ be the coarse moduli space of $\mathcal{X}$. The moduli space of semistable $G$-Higgs bundles on $\mathcal{X}$ was constructed by Simpson in \cite{Simp2010}, who showed this is, in fact, a quasi-projective scheme. The moduli problem and also moduli space of Higgs bundles on Deligne-Mumford stacks has been more generally studied in \cite{Sun191} and \cite{Sun201} by the second author. 

Suppose that $\mathcal{X}$ is a stacky curve, which is a smooth projective Deligne-Mumford stack of dimension one. We are thus considering the moduli space $\mathcal{M}_H(\mathcal{X},\alpha)$  of stable Higgs bundles over $\mathcal{X}$ with a fixed parabolic structure $\alpha$, as well as  $\mathcal{M}_H(\mathcal{X},G)$ and $\mathcal{M}_H(\mathcal{X},G,\alpha)$, the relative moduli spaces of pairs where the structure group of the underlying bundles is determined by a complex reductive algebraic group $G$. We show that the moduli space $\mathcal{M}_H(\mathcal{X},\alpha)$ is a Lie algebroid over the tangent space of the moduli space of stable bundles over $X$, thus implying the main theorem of this article:

\begin{thm}[Theorem \ref{404}]
	Let $\mathcal{X}$ be a stacky curve over $\mathbb{C}$. The moduli space $\mathcal{M}_H(\mathcal{X},\alpha)$ of stable Higgs bundles over $\mathcal{X}$ with fixed parabolic structure $\alpha$ admits a Poisson structure.
\end{thm}
In the course of developing the proof of this result, we highlight the importance of certain short exact sequences (Atiyah sequences) that arise. This opens the way for generalizing the above theorem in several directions. We first show similarly that the moduli space  $\mathcal{M}_H(\mathcal{X},G,\alpha)$ for a fixed faithful representation $G \hookrightarrow \text{GL}(V)$ is also equipped with a Poisson structure:

\begin{thm}[Theorem \ref{405}]
	The moduli space $\mathcal{M}_H(\mathcal{X},G,\alpha)$ of stable $G$-Higgs bundles over a stacky curve $\mathcal{X}$  with fixed parabolic structure $\alpha$ admits a Poisson structure.
\end{thm}

The notion of stability we consider here does not depend on the choice of a faithful representation $G \hookrightarrow \text{GL}(V)$. An alternative notion of (semi)stability for parabolic principal $G$-bundles and parabolic $G$-Higgs bundles is considered in the more recent work of Biquard, Garc\'{i}a-Prada and Mundet i Riera \cite{BiGaRi}. Furthermore, it is proven by the authors recently that the stability condition considered in this paper is equivalent to Ramanathan's stability condition of the corresponding logahoric Higgs torsor \cite{KSZ23}.

Note that in the special case of a root stack and a parabolic structure when all parabolic weights are rational, there is an alternative description of parabolic bundles as orbifold bundles (see \cite{Biswas2,FuSt,KSZ,NaSt}). Therefore, our theorems provide an orbifold version of the result of Bottacin \cite{Bot} and Markman \cite{Mar} in the case of simple pole divisors, as was first conjectured by Logares and Martens \cite[\S 5.2]{LoMa}.

Let $\widetilde{\mathcal{M}}$ be a moduli stack, and suppose that $\mathcal{M}$ is a fine moduli space of $\widetilde{\mathcal{M}}$. We have $\widetilde{\mathcal{M}} \cong {\rm Hom}(-,\mathcal{M})$, therefore, a Poisson structure on $\mathcal{M}$ will induce a Poisson structure on $\widetilde{\mathcal{M}}$. A theory of Poisson structures on stacks is needed here, and this was introduced in the dissertation of Waldron \cite{Wal}. Based on the theory of Poisson structures on stacks, we show that the moduli stack $\widetilde{\mathcal{M}}_H(\mathcal{X},\alpha)$ has a Poisson structure as a stack (Corollary \ref{408}).

Next, we show that the moduli space $\mathcal{M}_H(\mathcal{X},\mathcal{L},\alpha)$ of stable $\mathcal{L}$-twisted Higgs bundles over $\mathcal{X}$ with fixed parabolic structure $\alpha$ is Poisson, subject to the existence of a certain short exact sequence for any stable bundle $\mathcal{F} \in \mathcal{M}(\mathcal{X},\alpha)$; the precise statement is the following:

\begin{thm}[Theorem \ref{501}]\label{thm103}
	Let $\mathcal{X}$ be a stacky curve over $\mathbb{C}$ and let $X$ be the coarse moduli space of $\mathcal{X}$. Let $\alpha$ be a parabolic structure on $X$ and denote by $\bar{D}=p_1 + \dots+p_k \in X$ the divisor with respect to $\alpha$. Let $\mathbb{D}=q_1 +\dots + q_k \in \mathcal{X}$ be the corresponding divisor on $\mathcal{X}$, where $q_i$ is the point corresponding to $p_i$. If there exists a short exact sequence
	\begin{equation*}
	0 \rightarrow  \mathcal{H}om(\mathcal{F}\otimes \mathcal{L},\mathcal{F}  \otimes \omega_{\mathcal{X}} ) \rightarrow \mathcal{E}nd(\mathcal{F}) \rightarrow \mathfrak{n} \rightarrow 0
	\end{equation*}
	for any stable bundle $\mathcal{F} \in \mathcal{M}(\mathcal{X},\alpha)$,  such that
	\begin{enumerate}
		\item the morphism $\mathcal{H}om(\mathcal{F}\otimes \mathcal{L},\mathcal{F}  \otimes \omega_{\mathcal{X}} )\rightarrow \mathcal{E}nd(\mathcal{F})$ is not surjective,
		\item $\mathfrak{n}$ is a sheaf of Lie algebras supported on $\mathbb{D}$,
	\end{enumerate}
	then the moduli space $\mathcal{M}_H(\mathcal{X},\mathcal{L},\alpha)$ admits a Poisson structure.
\end{thm}

For $X$ a smooth projective curve, $\bar{D}=p_1 + \dots + p_k$ a reduced effective divisor on $X$ and $\bar{r}=(r_1,\dots,r_k)$ a $k$-tuple of positive integers, let $X_{\bar{D},\bar{r}}$ denote the corresponding root stack and consider the natural map $\pi: \mathcal{X} \rightarrow X$ from the root stack $\mathcal{X}$ to its coarse moduli space $X$. Denote by $\mathbb{D}$ the reduced divisor of $\pi^{-1}(\bar{D})$. The correspondence between parabolic bundles on $(X,\bar{D})$ and bundles on $\mathcal{X}:=X_{\bar{D},\bar{r}}$  implies the correspondence in the stability conditions as in \cite[Remarque 10]{Bor} and, more precisely, in moduli spaces $\mathcal{M}(\mathcal{X},\alpha)\cong \mathcal{M}^{par}(X,\alpha)$. It is natural to extend this correspondence to Higgs bundles \cite{BMW2, BMW,NaSt}, thus having $\mathcal{M}_H(\mathcal{X},\alpha)\cong \mathcal{M}_H^{spar}(X,\alpha)$. 

Let now $\mathcal{L}$ be a line bundle over $\mathcal{X}$ and denote by $L$ the parabolic line bundle over $(X,\bar{D})$ corresponding to $\mathcal{L}$, that is, a line bundle over $X$ together with a collection of flags $L_{p_i} \supset \{0\}$ with a real weight $0\leq \alpha_{p_i} <1$ for each point $p_i \in \bar{D}$, $i=1,...,k$. There is a one-to-one correspondence between $\mathcal{L}$-twisted Higgs bundles on $\mathcal{X}$ and $L$-twisted parabolic bundles on $(X,\bar{D})$ (see \cite[\S 5]{KSZ2}). We have the following proposition:

\begin{prop}[Proposition \ref{602}]
	Let $\mathcal{X}=X_{\bar{D},\bar{r}}$ be a root stack. Denote by $X$ the coarse moduli space of $\mathcal{X}$. The following statements hold:
	\begin{enumerate}
		\item There is an isomorphism $\mathcal{M}_H(\mathcal{X},\alpha)\cong \mathcal{M}^{spar}_H\left(X,\alpha\right)$, where $\mathcal{M}^{spar}_H\left(X,\alpha\right)$ is the moduli space of strongly parabolic Higgs bundles over $X$ with parabolic structure $\alpha$.
		\item There is an isomorphism $\mathcal{M}_H(\mathcal{X},\omega_{\mathcal{X}}(\mathbb{D}),\alpha)\cong \mathcal{M}^{par}_H\left(X,\alpha\right)$, where $\omega_{\mathcal{X}}(\mathbb{D})$ is the canonical line bundle over $\mathcal{X}$.
		\item Let $\mathcal{L}$ be an invertible sheaf on $\mathcal{X}$ and let $\pi:\mathcal{X}\to X$ be the map from the root stack to its coarse moduli space. Denote by $L$ the corresponding parabolic line bundle of $\mathcal{L}$ on $(X,\bar{D})$. Then there is an isomorphism $\mathcal{M}_H(\mathcal{X},\mathcal{L},\alpha)\cong \mathcal{M}^{par}_H\left(X,L,\alpha\right)$.
	\end{enumerate}
\end{prop}

Under these considerations, we find that the short exact sequence 
\begin{equation*}
0 \rightarrow  \mathcal{H}om(\mathcal{F}\otimes \mathcal{L},\mathcal{F}  \otimes \omega_{\mathcal{X}} ) \rightarrow \mathcal{E}nd(\mathcal{F}) \rightarrow \mathfrak{n} \rightarrow 0
\end{equation*}
in Theorem \ref{thm103} can be translated in the language of parabolic bundles by taking $\mathcal{L}=\omega_{\mathcal{X}}(\mathbb{D})$ to
\begin{equation*}
0\to \mathcal{SP}ar\mathcal{E}nd(F)\to \mathcal{P}ar\mathcal{E}nd(F)\to \mathfrak{n} \to 0,
\end{equation*}
where $F$ is the corresponding parabolic bundle of $\mathcal{F}$. With respect to the above observation, Theorem \ref{thm103} gives an alternative proof to \cite{LoMa} on the existence of a Poisson structure on $\mathcal{M}_H^{par}(X,\alpha)$, the moduli space of stable parabolic Higgs bundles with parabolic structure $\alpha$ on $(X,\bar{D})$, where $X$ is an irreducible smooth curve and $\bar{D}$ is a reduced effective divisor.

The method via Lie algebroids used in this article provides the construction of Poisson structures on a wide class of Higgs bundle moduli spaces over stacky curves. In order to further investigate completely integrable systems embedded as symplectic leaves of these Poisson moduli spaces, an explicit construction of a canonical moment map would be required giving a Hamiltonian $G$-stack, so that the symplectic leaves are well-defined via an appropriate Marsden--Weinstein symplectic reduction theorem for symplectic stacks. We hope to explicitly demonstrate this more direct approach in a future article.

\section{Preliminaries}
In this preliminary section, we collect the necessary background on stacks that we shall need for our purposes, and we also consider parabolic bundles and compare them to bundles on root stacks. In this paper, all stacks and schemes are defined over $\mathbb{C}$. The interested reader may refer to \cite{BMW,Ol} for further background on the material covered in \S 2.1 - 2.3. A good reference for \S 2.4 is \cite[\S4, \S 5]{Simp2010}. In \S 2.5, we give the definition of parabolic structures of $G$-bundles, which depends on a fixed faithful representation $G \hookrightarrow \text{GL}(V)$.

\subsection{Deligne-Mumford Stacks}
A \emph{Deligne-Mumford stack} $\mathcal{X}$ is an algebraic stack such that there exists a surjective \'{e}tale morphism $U \rightarrow \mathcal{X}$, where $U$ is a $\mathbb{C}$-scheme. The data $(U,u)$ is called a \emph{chart} of $\mathcal{X}$, where $u$ is surjective. If $u: U \rightarrow \mathcal{X}$ is an \'{e}tale morphism (not necessarily surjective), the pair $(U,u)$ is called a \emph{local chart} of $\mathcal{X}$. Let $(U,u)$ and $(V,v)$ be two charts of $\mathcal{X}$. A \emph{morphism of charts} $(U,u)$ and $(V,v)$ is a morphism $f_{uv}:(U,u) \rightarrow (V,v)$ of schemes such that the following diagram commutes
\begin{center}
	\begin{tikzcd}
	U \arrow[rd,swap, "u"] \arrow[rr, "f_{uv}"] &  & V \arrow[ld,"v"] \\
	& \mathcal{X} &
	\end{tikzcd}
\end{center}
If $\mathcal{X}$ is locally of finite type and with finite diagonal, then there exists a coarse moduli space $X$ (as an algebraic space) of $\mathcal{X}$ \cite[Theorem 11.1.2]{Ol}. Denote by $\pi: \mathcal{X} \rightarrow X$ the natural morphism.

\begin{defn}
	A Deligne-Mumford stack $\mathcal{X}$ is \emph{smooth and projective} if it satisfies the following conditions:
	\begin{enumerate}
		\item there exists a surjective \'etale morphism $Y \rightarrow \mathcal{X}$ such that $Y$ is a smooth projective variety;
		\item $\mathcal{X}$ can be written as a global quotient $[Y \backslash \Gamma]$, where $\Gamma$ is a finite group;
		\item $\mathcal{X}$ has a coarse moduli space $X$, which is a smooth projective variety,
	\end{enumerate}
	and we say $\mathcal{X}$ is a \emph{smooth projective Deligne-Mumford stack}. Furthermore, if $\mathcal{X}$ is connected and of dimension one, we say that $\mathcal{X}$ is a stacky curve. 
\end{defn}

The idea of smooth projective Deligne-Mumford stacks is introduced by Simpson to establish a stacky version of the nonabelian Hodge correspondence \cite[Theorem 5.4]{Simp2010}. Compared to Simpson's definition, we also require that the coarse moduli space $X$ is smooth and projective for the purposes of this paper. 

\begin{defn}\label{204}
	Let $\mathcal{X}$ be a smooth projective Deligne--Mumford stack. A \emph{Cartier divisor} $D$ on $\mathcal{X}$ is defined on each chart $(U,u)$ as follows: there is a Cartier divisor $D_u$ on $U$ such that if $f_{uv}: (U,u) \rightarrow (V,v)$ is a morphism of charts, then $f^*_{uv}(D_v)=D_u$. We say a Cartier divisor $D$ has \emph{normal crossings} if for each chart $(U,u)$, the divisor $D_u$ has normal crossings.
\end{defn}

\subsection{Sheaves}
\subsubsection*{\textbf{Sheaves on Stacks}}
Let $\mathcal{X}$ be a Deligne-Mumford stack and let $(U,u)$ be a chart of $\mathcal{X}$. Instead of giving the precise definition of coherent sheaves on $\mathcal{X}$, we give a more workable definition. A \emph{coherent sheaf} $\mathcal{F}$ on $\mathcal{X}$ is defined to be a pair $(F,\sigma)$ such that $F$ is a coherent sheaf on $U$ and $\sigma: s^* F \xrightarrow{\cong} t^* F$ is an isomorphism, where 
\begin{align*}
s,t:U \times_{\mathcal{X}} U \rightrightarrows U
\end{align*}
are the source and target maps. In fact, this definition does not depend on the choice of the chart we take. We refer the reader to \cite[Chapter 7]{Ol} for more details. A coherent sheaf $\mathcal{F}$ is \emph{locally free} if the coherent sheaf $F$ in the corresponding pair $(F,\sigma)$ is locally free.

We next provide some examples of coherent sheaves on $\mathcal{X}$, and we omit the isomorphism $\sigma$ for simplicity. The \emph{structure sheaf $\mathcal{O}_{\mathcal{X}}$} is defined as $\mathcal{O}_U$ on the chart $(U,u)$. Let $\Omega^1_{\mathcal{X}}$ be the \emph{cotangent sheaf} on $\mathcal{X}$ defined on the local chart $(U,u)$ by $\Omega^1_{U}$.

\subsubsection*{\textbf{Higgs Bundles over stacks}}
Let $\mathcal{X}$ be a stacky curve. A \emph{Higgs bundle} on $\mathcal{X}$ is a pair $(\mathcal{F},\Phi)$, where $\mathcal{F}$ is a locally free sheaf on $\mathcal{X}$ and $\Phi: \mathcal{F} \rightarrow \mathcal{F} \otimes \Omega^1_{\mathcal{X}}$ is a morphism called a \emph{Higgs field}. Given a line bundle $\mathcal{L}$ on $\mathcal{X}$, an \emph{$\mathcal{L}$-twisted Higgs bundle} over $\mathcal{X}$ is a pair $(\mathcal{F},\Phi)$, where $\mathcal{F}$ is a locally free sheaf over $\mathcal{X}$ and $\Phi:\mathcal{F} \rightarrow \mathcal{F} \otimes \mathcal{L}$ is a morphism called an \emph{$\mathcal{L}$-twisted Higgs field}.

\subsubsection*{\textbf{Stability Condition}}
Let $\mathcal{X}$ be a stacky curve. Let $\pi: \mathcal{X} \rightarrow X$ be the natural map to its coarse moduli space. Denote by $q: \mathcal{X} \rightarrow \mathbb{C}$ the structure morphism. Let $\mathcal{F}$ be a locally free sheaf over $\mathcal{X}$. The \emph{degree} of $\mathcal{F}$ over $\mathcal{X}$ is defined (see \cite[\S 4.1]{Bor}) as
\begin{align*}
{\rm deg}(\mathcal{F}):=q_*(c_1(\mathcal{F})),
\end{align*}
where $c_1(\mathcal{F})$ is the first Chern class of $\mathcal{F}$. A locally free sheaf $\mathcal{F}$ is called \emph{semistable} (resp. \emph{stable}), if for any subsheaf $\mathcal{F}'$ with ${\rm rk}(\mathcal{F}') < {\rm rk}(\mathcal{F})$, it is
\begin{align*}
\frac{{\rm deg}(\mathcal{F}')}{ {\rm rk}(\mathcal{F}') } \leq \frac{{\rm deg}(\mathcal{F})}{ {\rm rk}(\mathcal{F}) } \quad (\text{resp. } <).
\end{align*}
Similarly, a Higgs bundle $(\mathcal{F},\Phi)$ is called \emph{semistable} (resp. \emph{stable}), if for any $\Phi$-invariant subsheaf $\mathcal{F}'$ with ${\rm rk}(\mathcal{F}') < {\rm rk}(\mathcal{F})$, it is
\begin{align*}
\frac{{\rm deg}(\mathcal{F}')}{ {\rm rk}(\mathcal{F}') } \leq \frac{{\rm deg}(\mathcal{F})}{ {\rm rk}(\mathcal{F}) } \quad (\text{resp. } <).
\end{align*}
Recall that $\Phi$-invariant subsheaf $\mathcal{F}'$ means that $\Phi(\mathcal{F}')\subseteq \mathcal{F}' \otimes \Omega^1_{\mathcal{X}}$.

\subsection{Principal bundles}
\subsubsection*{\textbf{Principal Bundles and $G$-Higgs bundles}}
Let $\mathcal{X}$ be a Deligne-Mumford stack, and we fix a chart $(U,u)$ of $\mathcal{X}$. Let $G$ be a connected complex reductive algebraic group. Denote by $\mathfrak{g}$ the Lie algebra of $G$. A \emph{$G$-bundle} $\mathcal{E}$ on $\mathcal{X}$ is defined as a pair $(E,\sigma)$, where $E$ is a $G$-bundle on $U$ and $\sigma: s^* E \xrightarrow{\cong} t^* E$ is an isomorphism. The following equivalence \cite[Proposition 1.2]{BMW} is well-known:
\begin{align*}
{\rm Bun}_{G} (\mathcal{X}) \cong \underline{{\rm Hom}}(\mathcal{X}, BG),
\end{align*}
where ${\rm Bun}_{G}(\mathcal{X})$ is the category of $G$-bundles fibered in groupoids on $\mathcal{X}$ and $BG$ is the classifying stack of $G$.

Now let $\mathcal{X}$ be a stacky curve. A \emph{$G$-Higgs bundle} on $\mathcal{X} $ is a pair $(\mathcal{E},\Phi)$ such that $\mathcal{E}$ is a principal $G$-bundle on $\mathcal{X}$ and $\Phi \in H^0(\mathcal{X},\mathcal{E}(\mathfrak{g}) \otimes \Omega^1_{\mathcal{X}} )$ is a section, where $\mathcal{E}(\mathfrak{g}):= \mathcal{E} \times_G \mathfrak{g}$ is the adjoint bundle.

\subsubsection*{\textbf{Stability Condition}}
Let $\mathcal{X}$ be a stacky curve, and we write it as a global quotient $\mathcal{X}=[Y/\Gamma]$. By definition, a $G$-Higgs bundle $\mathcal{E}$ on $\mathcal{X}$ is equivalent to a $\Gamma$-equivariant $G$-Higgs bundle on $Y$ (see \cite[\S 2]{BaBisNa}). We fix a faithful representation $G \hookrightarrow \text{GL}(V)$. Denote by $\mathcal{E}(V):=\mathcal{E} \times_G V$ the associated bundle. In this case, a Higgs field $\Phi \in H^0(\mathcal{X},\mathcal{E}(\mathfrak{g}) \otimes \Omega^1_{\mathcal{X}})$ corresponds to an element in $H^0(\mathcal{X}, \mathcal{E}nd(\mathcal{E}(V)) \otimes \Omega^1_{\mathcal{X}})$. A $G$-Higgs bundle $(\mathcal{E},\Phi)$ is \emph{semistable} (resp. stable), if the associated Higgs bundle $(\mathcal{E}(V),\Phi)$ is semistable (resp. stable). The semistability condition of $G$-Higgs bundles is called \emph{semiharmonic} in Simpson's papers \cite[page 49]{Simp3} and \cite[\S 7]{Simp2010}.

\begin{rem}
	Indeed, there are several ways to define the semistability of principal $G$-bundles. In this paper, we shall use the above definition for semistability, since it has been proved that the moduli space of $G$-Higgs bundles over a projective Deligne-Mumford stack exists in this case \cite{Simp2010}.
\end{rem}

\subsection{Root Stacks}\label{rootsta}
Root stacks are a highly significant case of Deligne-Mumford stacks. It is known that a smooth projective Deligne-Mumford stack is locally isomorphic to a root stack (see \cite{Simp2010}). We shall review this result in this subsection.

Let $X$ be a smooth projective variety, and let $L$ be a line bundle on $X$. Note that the following categories are equivalent
\begin{align*}
\{\text{invertible sheaves on } X\} \longleftrightarrow \{\text{morphisms: } X \rightarrow B\mathbb{G}_m\}.
\end{align*}

Let $s \in \Gamma(X,L)$ be a section of $L$. The pair $(L,s)$ defines a morphism $X \rightarrow [\mathbb{A}^1/\mathbb{G}_m]$. The category of pairs $(L,s)$, where $L$ is a line bundle on $X$ and $s \in \Gamma(X,L)$, and the category of morphisms $X \rightarrow [\mathbb{A}^1/\mathbb{G}_m]$ are equivalent. This equivalence can be generalized to $n$ line bundles and $n$ sections. More precisely, the category of morphisms $X \rightarrow [\mathbb{A}^n/\mathbb{G}^n_m]$ is equivalent to the category of $n$-tuples $(L_i,s_i)^n_{i=1}$, where $L_i$ is a line bundle on $X$ and $s_i \in \Gamma(X,L_i)$ (see \cite[Lemma 2.1.1]{Cad}). 

Let $\theta_r: [\mathbb{A}^1/\mathbb{G}_m] \rightarrow [\mathbb{A}^1/\mathbb{G}_m]$ be the morphism induced by $r$-th power maps on both $\mathbb{A}^1$ and $\mathbb{G}_m$. Let $X_{(L,s,r)}$ be the fiber product
$X \times_{[\mathbb{A}^1/\mathbb{G}_m], \theta_r}[\mathbb{A}^1/\mathbb{G}_m]$. The stack $X_{(L,s,r)}$ is then called the \emph{$r$-th root stack}. Let $\bar{D}=(D_1,\dots,D_k)$ be a $k$-tuple of effective Cartier divisors $D_i \subseteq X$. Let $\bar{r}=(r_1,\dots,r_k)$ be a $k$-tuple of positive integers. We define $\theta_{\bar{r}}: [\mathbb{A}^k/\mathbb{G}^k_m] \rightarrow [\mathbb{A}^k/\mathbb{G}^k_m]$ to be the morphism $\theta_{r_1} \times \dots \times \theta_{r_k}$. The \emph{Cadman-Vistoli root stack $X_{\bar{D},\bar{r}}$} is defined as the fiber product $X \times_{[\mathbb{A}^k/\mathbb{G}^k_m], \theta_{\bar{r}}}[\mathbb{A}^k/\mathbb{G}^k_m]$, where the morphism $X \rightarrow [\mathbb{A}^k/\mathbb{G}^k_m]$ is defined by $(\mathcal{O}(D_i),s_{D_i})^k_{i=1}$ (see \cite[Definition 2.2.4]{Cad}).

Let $\mathcal{X}$ be a Deligne-Mumford stack locally of finite presentation with finite diagonal. Denote by $X$ its coarse moduli space. Let $x$ be a geometric point of $\mathcal{X}$, that is, a morphism $x: {\rm Spec} \,\mathbb{C} \rightarrow \mathcal{X}$. Denote by $G_x$ the automorphism group of $x$. Since $\mathcal{X}$ is a Deligne-Mumford stack, $\Gamma_x$ is a finite group. Let $\bar{x}$ be the corresponding point of $x$ in the coarse moduli space. There exists (see \cite[Theorem 11.3.1]{Ol}) a neighborhood $(U,u)$ of $\bar{x}$ and a finite morphism $V \rightarrow U$ such that
\begin{align*}
\mathcal{X}\times_X U \cong [V/ \Gamma_x].
\end{align*}
This property tells us that a Deligne-Mumford stack is locally a quotient stack. More precisely, let $x \in \mathcal{X}$ be a geometric point, and let $\bar{x}$ be its corresponding point in $X$. There is an \'{e}tale neighborhood $(U,u)$ of $\bar{x} \in X$ such that
\begin{align*}
\mathcal{X} \times_{X} U \cong U_{\bar{D},\bar{r}}
\end{align*}
for some $\bar{D}$ and $\bar{r}$ (see \cite[\S 4 and \S 5]{Simp2010}).

\subsection{Parabolic Bundles}
Now we consider a special case of a root stack. Let $X$ be a smooth projective curve and let $\bar{D}=(D_1,\dots,D_k)$ be a $k$-tuple of divisors $D_i \subseteq X$ such that $D_i=p_i$ is a single point. Let $\bar{r}=(r_1,\dots,r_k)$ be a $k$-tuple of positive integers. The notation $r(p)$ shall refer to the integer in $\bar{r}$ corresponding to the point $p \in \bar{D}$.

\subsubsection*{\textbf{Parabolic Structures and Parabolic Bundles}}
We assume that $X_{\bar{D},\bar{r}}$ can be written as a global quotient $[U/\Gamma] \cong X_{\bar{D},\bar{r}}$, where $U$ is a smooth projective variety. There is a natural map $\pi: U \rightarrow X$. A locally free sheaf $\mathcal{F}$ of rank $n$ over $X_{\bar{D},\bar{r}}$ is equivalent to a locally free sheaf $F$ on $U$ together with a local trivialization $\Theta_p : F_{p} \rightarrow U_p \times \mathbb{C}^n$ for each point $p \in \bar{D}$, where $U_p$ is a neighborhood of $\pi^{-1}(p)$ and $F_p:=F|_{U_p}$, such that $\Theta_p$ is $\mathbb{Z}_{r(p)}$-equivariant with respect to the following action around the point $p \in \bar{D}$
\begin{align*}
t(z;z_1,z_2,...,z_n)=(tz;t^{\alpha'_1(p)}z_1,t^{\alpha'_2(p)}z_2,...,t^{\alpha'_n(p)}z_n),
\end{align*}
where $\alpha'_1(p),...,\alpha'_n(p)$ are integers such that $0 \leq \alpha'_1(p) \leq \alpha'_2(p) \leq ... \leq \alpha'_n(p) < r(p)$. We can take local holomorphic sections $f_1,...,f_n$ of $F$ such that $\{f_1(p),...,f_n(p)\}$ is a basis of $F_p$ consisting of eigenvectors. Then, we set
\begin{align*}
\Theta=(t^{-\alpha'_1(p)}(t \cdot f_1),...,t^{-\alpha'_n(p)}(t \cdot f_n)),
\end{align*}
where $t \cdot f_i(x)=t^{\alpha'_i(p)}f_i(x)$. Suppose that there are $l(p)$ distinct values in $\{\alpha'_1(p),\dots,\alpha'_n(p)\}$, and let $\{\alpha_1(p),\dots,\alpha_{l(p)}(p)\}$ be the $l(p)$ distinct values such that
\begin{align*}
    0 \leq \alpha_1(p) < \dots < \alpha_{l(p)}(p) < 1.
\end{align*}
Then, we define a weighted filtration of $F$ on $\pi^{-1}(p)$,
\begin{align*}
F|_{\pi^{-1}(p)} =& F_1(p) \supseteq \dots  \supseteq F_{l(p)}(p) \supseteq 0,\\
& \frac{\alpha_1(p)}{r(p)} < \dots < \frac{\alpha_{l(p)}(p)}{r(p)},
\end{align*}
where $F_i(p) / F_{i+1}(p)$ corresponds to the subspace, on which the action of $t$ is given by $t^{\alpha_i(p)}$. The weighted filtration associated to each puncture determines a \emph{parabolic structure} on $F$, and we prefer to use the notation $\alpha$ for it.

A \emph{parabolic bundle} on $(X,\bar{D})$ is a pair $(F,\alpha)$, where $F$ is a locally free sheaf and $\alpha$ is a parabolic structure over each of the points in $\bar{D}$. A locally free sheaf $\mathcal{F}$ on $X_{\bar{D},\bar{r}}$ is equivalent to a parabolic bundle on $(X,\bar{D})$. We refer the reader to \cite{Biswas2,FuSt,NaSt} for more details on this correspondence. Furthermore, the parabolic structure itself is an important topological invariant of a locally free sheaf over $X_{\bar{D},\bar{r}}$, which can be used in describing the connected components of the moduli space of locally free sheaves on $\mathcal{X}$ \cite{KSZ}. With respect to this correspondence, we say that a locally free sheaf $\mathcal{F}$ on $X_{\bar{D},\bar{r}}$ has \emph{parabolic structure $\alpha$} if the corresponding parabolic bundle on $X$ is of parabolic type $\alpha$.

A parabolic structure of a locally free sheaf over a point $p$ corresponds to a parabolic group. A parabolic structure over a point is given by integers $\alpha'_1(p) \leq \dots \leq \alpha'_n(p)$, which defines a \emph{type}. For example, the sequence of integers $(1,1,3,3,3,3,4,4,4)$ gives us the type $(2,4,3)$. This type in turn uniquely determines a parabolic subgroup of ${\rm GL}_n$:
$\begin{pmatrix}
A_1 & * & * \\
0   & A_2 & *\\
0  & 0   & A_3
\end{pmatrix},$ where $A_1$ is a $2$ by $2$ matrix, $A_2$ is a $4$ by $4$ matrix and $A_3$ is a $3$ by $3$ matrix.

The correspondence between parabolic structures and parabolic groups can be also understood from Higgs fields. Let now $\Phi$ be a Higgs field of $\mathcal{F}$. With respect to the above setup, $\Phi$ can be written as follows around $p \in \bar{D}$:
\begin{align*}
\Phi=(\Phi_{ij})_{1 \leq i,j \leq l(p)},
\end{align*}
where
\begin{align*}
\Phi_{ij}=
\begin{cases}
z^{\alpha_{i}(p)-\alpha_{j}(p)} \hat{\Phi}_{ij}(z^{r(p)})\frac{dz}{z} & \mathrm{ if } \alpha_i \geq \alpha_j\\
0 & \mathrm{ if } \alpha_i < \alpha_j,
\end{cases}
\end{align*}
and $\hat{\Phi}_{ij}$ are block matrices of holomorphic functions on $F$. Note that the morphism $\Phi$ can be regarded as an element in the parabolic subgroup. In fact, the calculation also works for any endomorphism of $\mathcal{F}$ (see \cite{KSZ,NaSt}).

\subsubsection*{\textbf{Parabolic Degree and Stability Condition}}
The \textit{parabolic degree} of a parabolic vector bundle $F$ is given by
\[\pardeg(F)=\deg(F)+\sum_{p\in D}\sum_{i=1}^{l(p)} \alpha_i(p) \cdot \dim (F_i(p)/F_{i+1}(p)) = \deg(F) + \sum_{p\in D}\sum_{i=1}^{n} \alpha'_i(p).\]
We call a parabolic vector bundle $F$ \textit{semistable} (resp. \textit{stable}) if for all parabolic subbundles $F'$, we have
\[\frac{\pardeg(F')}{\mathrm{rk}(F')}\le \frac{\pardeg(F)}{\mathrm{rk}(F)} \quad (\text{resp. } <).\]
From the correspondence between parabolic bundles $F$ over $X$ and bundles $\mathcal{F}$ over $X_{\bar{D},\bar{r}}$, we have
\begin{align*}
\deg(\mathcal{F})=\pardeg(F).
\end{align*}
Note that this property also provides that $F$ is semistable (resp. stable) if and only if $\mathcal{F}$ is semistable (resp. stable).

\subsubsection*{\textbf{Parabolic Higgs Bundles}}
Let $E$ be a parabolic bundle on $(X,\bar{D})$. Let $\omega_X$ be the canonical line bundle of $X$. A \emph{parabolic Higgs field $\Phi$} is a section $H^0(X,\mathcal{E}nd(E) \otimes \omega_X(\bar{D}))$, which preserves the filtration on each puncture $p \in \bar{D}$, i.e. $\Phi|_p (F_i(p)) \subseteq F_i(p) \otimes \omega_X(\bar{D})$. If $\phi|_p (F_i(p)) \subseteq F_{i+1}(p) \otimes \omega_X(\bar{D})$, it is called a \emph{strongly parabolic Higgs field}. Roughly speaking, on each puncture, parabolic Higgs fields can be regarded as elements in the parabolic subgroup, while strongly parabolic Higgs fields are regarded as elements in the unipotent subgroup of the parabolic subgroup. We refer the reader to \cite{LoMa,Yoko2} for more details.

We denote the sheaf of parabolic homomorphisms between two parabolic vector bundles $E$ and $F$ by $\mathcal{P}ar\mathcal{H}om(E,F)$, the sheaf of strongly parabolic homomorphisms by $\mathcal{SP}ar\mathcal{H}om(E,F)$. In addition, we denote $\mathcal{P}ar\mathcal{E}nd(E)=\mathcal{P}ar\mathcal{H}om(E,E)$ and $\mathcal{SP}ar\mathcal{E}nd(E)=\mathcal{SP}ar\mathcal{H}om(E,E)$. We now define:

\begin{enumerate}
	\item  A \textit{parabolic Higgs bundle} over $(X,\bar{D})$ is a pair $(E,\Phi)$, where $E$ is a parabolic bundle and $\Phi\in H^0(X,\mathcal{P}ar\mathcal{E}nd(E)\otimes \omega_X(\bar{D}))$. We call it \emph{stable} (resp. semistable) if it is stable (resp. semistable) with respect to the $\Phi$-invariant subbundles.
	\item A \textit{strongly parabolic Higgs bundle} $(X,\bar{D})$ is a pair $(E,\Phi)$ where $E$ is a parabolic bundle and $\Phi\in H^0(X,\mathcal{SP}ar\mathcal{E}nd(E)\otimes \omega_X(\bar{D}))$.
	\item Let $L$ be a parabolic line bundle on $X$. We call \textit{$L$-twisted parabolic Higgs bundle} over $(X,\bar{D})$ a pair $(E,\Phi)$, where $E$ is a parabolic bundle and $\Phi\in H^0(X,\mathcal{P}ar\mathcal{E}nd(E)\otimes L)$.
\end{enumerate}

\subsubsection*{\textbf{Parabolic Structure of Principal Bundles}}
Closing this subsection, we give the definition of the parabolic structure of a $G$-bundle. We fix a faithful representation $G\hookrightarrow \text{GL}(V)$. Let $\mathcal{E}$ be a $G$-bundle on $\mathcal{X}$ and denote by $\mathcal{E}(V)$ the associated bundle. We say that the \emph{parabolic structure} of $\mathcal{E}$ is $\alpha$, if the parabolic structure of the corresponding associated bundle $\mathcal{E}(V)$ is $\alpha$. Although the parabolic structure of a principal $G$-bundle depends on the choice of the faithful representation in this definition, the stability condition of the associated bundle is equivalent to Ramanathan's stability condition for the corresponding parabolic $G$-Higgs bundle (or logahoric Higgs torsor). This property is recently studied and proven by the authors in \cite{KSZ23}.

\section{Deformation Theory on Moduli Spaces of Higgs Bundles over Deligne-Mumford Stacks}
In this section, we review some results on the deformation theory of moduli spaces of Higgs bundles over Deligne-Mumford stacks, which will help us calculate the tangent space of the moduli spaces we are interested in. In \S 3.1, we define all the moduli spaces we consider in this paper and in \S 3.2 we review the deformation theory on those moduli spaces. In \S 3.3, we review the Grothendieck duality of coherent sheaves over Deligne-Mumford stacks, while in \S 3.4 we restrict to stacky curves and apply the results from \S 3.2 and \S 3.3 to construct a morphism
$T^*(\mathcal{M}_H(\mathcal{X})) \rightarrow T(\mathcal{M}_H(\mathcal{X}))$, which will be used in order to construct a Poisson structure on $\mathcal{M}_H(\mathcal{X})$ later on in \S 4.

\subsection{Moduli Space of Higgs Bundles on Smooth Projective Deligne-Mumford Stacks}
Let $\mathcal{X}$ be a smooth projective Deligne-Mumford stack. We first review the process of constructing the moduli space $\mathcal{M}_H(\mathcal{X})$ (see \cite{Simp2010} for more details).

Let $Y \rightarrow \mathcal{X}$ be a surjective \'etale morphism such that $Y$ is a smooth projective variety, and this admits a proper hyper-covering by smooth projective varieties \cite[Theorem 5.8]{Simp2010}. In other words, there is a simplicial resolution of $\mathcal{X}$ by smooth projective varieties. We briefly review next the construction of a simplicial resolution of $\mathcal{X}$. The first step of this construction is given by the existence of a surjective \'etale morphism $Y_0 \rightarrow \mathcal{X}$, where $Y_0:=Y$. Then we look at $Y_0 \times_{\mathcal{X}} Y_0$. By resolving singularities, we get a smooth projective variety $Y_1$. This provides the starting point of a simplicial resolution $Y_1 \rightrightarrows Y_0 \rightarrow \mathcal{X}$. Iterating the process, we get the simplicial resolution $Y_{\bullet}$ of $\mathcal{X}$. Since each $Y_k$ is a smooth projective variety over $\mathbb{C}$, the moduli space $\mathcal{M}_{H}(Y_k)$ of stable Higgs bundles over $Y_k$ exists \cite[Theorem 4.7]{Simp2}. Thus, there is a natural way to construct the moduli space $\mathcal{M}_{H}(Y_{\bullet})$ of semistable Higgs bundles over $Y_{\bullet}$ \cite[\S 6]{Simp2010}. Indeed, the moduli space of stable (resp. semistable) Higgs bundles over $Y_{\bullet}$ is isomorphic to the moduli space of stable (resp. semistable) Higgs bundles over $\mathcal{X}$ \cite[\S 9]{Simp2010}:
\begin{align*}
\mathcal{M}_{H}(Y_{\bullet}) \cong \mathcal{M}_{H}(\mathcal{X}).
\end{align*}
The moduli space of Higgs bundles over $\mathcal{X}$ is proved to be a quasi-projective scheme \cite[\S 6]{Simp2010}.

In this paper, we prefer to consider the stable locus of the moduli space, but some of our results can be extended to the semistable case. We use the following notations for the moduli spaces we consider:
\begin{itemize}
	\item $\mathcal{M}(\mathcal{X},\bullet_1,\bullet_3)$: the moduli space of stable bundles on $\mathcal{X}$,
	\item $\mathcal{M}_H(\mathcal{X},\bullet_1,\bullet_2,\bullet_3)$: the moduli space of stable Higgs bundles on $\mathcal{X}$,
	\item $\mathcal{M}^{par}(X,\bullet_1,\bullet_3)$: the moduli space of stable parabolic bundles on $X$,
	\item $\mathcal{M}_H^{par}(X,\bullet_1,\bullet_2,\bullet_3)$: the moduli space of stable parabolic Higgs bundles on $X$,
\end{itemize}
where $\bullet_1$ is the position for the structure group $G$, $\bullet_2$ is for the line bundle $\mathcal{L}$ (as the twisting bundle) and $\bullet_3$ is for the parabolic structure $\alpha$. For example, $\mathcal{M}_H(\mathcal{X},G,\alpha)$ denotes the moduli space of stable $G$-Higgs bundles on $\mathcal{X}$ with parabolic structure $\alpha$. The moduli spaces over Deligne-Mumford stacks are constructed in \cite{Simp2010}, while their parabolic analogs were constructed in \cite{Yoko1}.

\subsection{Deformation Theory}
The goal of this subsection is to calculate the tangent space of $\mathcal{M}_H(\mathcal{X})$, which is the moduli space of Higgs bundles on $\mathcal{X}$.

The moduli space $\mathcal{M}_{H}(\mathcal{X})$ represents the following moduli problem
\begin{align*}
\widetilde{\mathcal{M}}_{H}(\mathcal{X}): (\text{Sch/$\mathbb{C}$})^{{\rm op}} \rightarrow \text{Set}
\end{align*}
such that for each $\mathbb{C}$-scheme $T$, $\widetilde{\mathcal{M}}_{H}(\mathcal{X})(T)$ is the set of isomorphism classes of $T$-flat families of stable Higgs bundles $(\mathcal{F},\Phi)$. In this section, we shall be using the notation $\widetilde{\mathcal{M}}$ for the moduli problem $\widetilde{\mathcal{M}}_H(\mathcal{X})$ and $\mathcal{M}$ for the moduli space of stable Higgs bundles. 

\begin{rem}
	A moduli problem is usually defined as a functor, which is taken as the first step to construct moduli spaces \cite{HarMor}. This functor can be improved to be a category fibered in groupoids, and therefore, a ``moduli problem" can be equipped with a stack structure. In this paper, we use the same notation $\widetilde{\mathcal{M}}$ for moduli problems and the corresponding stacks. We refer the reader to \cite{CasaWise} for more details.
\end{rem}

Let $\text{Spec}(A)$ be an affine scheme, and let $M$ be an $A$-module. Let $\xi=(\mathcal{F},\Phi)$ be an element in $\widetilde{\mathcal{M}}(A)$, where $\widetilde{\mathcal{M}}(A):=\widetilde{\mathcal{M}}(\text{Spec}(A))$. There is a natural map \begin{align*}
\widetilde{\mathcal{M}}(A[M]) \rightarrow \widetilde{\mathcal{M}}(A).
\end{align*}
Denote by $\widetilde{\mathcal{M}}_{\xi}(A[M])$ the pre-image of the element $\xi \in \widetilde{\mathcal{M}}(A)$. In other words, $\widetilde{\mathcal{M}}_{\xi}(A[M])$ is the set of elements whose restriction to $\mathcal{X}_A$ is $\xi$. The set $\widetilde{\mathcal{M}}_{\xi}(A[M])$ is known as the set of \emph{deformations} of $\xi$ with respect to the extension
\begin{align*}
0\rightarrow M \rightarrow A[M] \rightarrow A \rightarrow 0.
\end{align*}
Now let $A=\mathbb{C}$ and let $M$ be the free rank one $A$-module generated by $\varepsilon$. We consider the short exact sequence
\begin{align*}
0\rightarrow (\varepsilon) \rightarrow \mathbb{C}[\varepsilon] \rightarrow \mathbb{C} \rightarrow 0,
\end{align*}
where $\mathbb{C}[\varepsilon]$ is the ring $\mathbb{C}[\varepsilon]/(\varepsilon^2)$ and we abuse the notation here.

Given an element $\xi \in \widetilde{\mathcal{M}}(\mathbb{C})$, an \emph{infinitesimal deformation} of $\xi$ is an element in $\widetilde{\mathcal{M}}_{\xi}(\mathbb{C}[\varepsilon])$. It is well-known that the set of all infinitesimal deformations of $\xi$ is the tangent space of the moduli space $\mathcal{M}$ at the point $\xi$.

Let $\xi=(\mathcal{F},\Phi)$ be an element in $\widetilde{\mathcal{M}}(A)$. The deformation complex $C_M^{\bullet}(\mathcal{F},\Phi)$ is defined as
\begin{align*}
C_M^{\bullet}(\mathcal{F},\Phi): C_M^0(\mathcal{F}) = \mathcal{E}nd(\mathcal{F}) \otimes M \xrightarrow{e(\Phi)} C_M^1(\mathcal{F}) = \mathcal{E}nd(\mathcal{F}) \otimes \Omega^1_{\mathcal{X}} \otimes M,
\end{align*}
where the map $e(\Phi)$ is given by
\begin{align*}
e(\Phi)(s)=-\rho(s)(\Phi)
\end{align*}
and $p_{\mathcal{X}}: \mathcal{X} \times_{\mathbb{C}} A \rightarrow \mathcal{X}$ is the natural projection. If there is no ambiguity, we omit the symbols $M$, $\mathcal{F}$, $\Phi$ and use the notation
\begin{align*}
C^{\bullet}: C^0 = \mathcal{E}nd(\mathcal{F}) \otimes M \xrightarrow{e(\Phi)} C^1 = \mathcal{E}nd(\mathcal{F}) \otimes \Omega^1_{\mathcal{X}} \otimes M
\end{align*}
for the deformation complex. Biswas and Ramanan in \cite{BisRam} first described the set of infinitesimal deformations of a Higgs bundle over a smooth projective variety and proved that the set of deformations is isomorphic to the first hypercohomology of a two-term complex. This approach was generalized to the case of Higgs bundles over a Deligne-Mumford stack in \cite{Sun191}. Therefore, we have the following proposition:

\begin{prop}[Proposition 3.3 in \cite{Sun191}]\label{301}
	Let $\xi =(\mathcal{F},\Phi)$ be a Higgs bundle in $\widetilde{\mathcal{M}}(A)$. The set of deformations $\widetilde{\mathcal{M}}_{\xi}(A[M])$ is isomorphic to the hypercohomology group $\mathbb{H}^1(C^{\bullet})$, where $C^{\bullet}$ is the complex
	\begin{align*}
	C^{\bullet}:C^0 =\mathcal{E}nd(\mathcal{F}) \otimes M \xrightarrow{e(\Phi)} C^1 = \mathcal{E}nd(\mathcal{F}) \otimes \Omega^1_{\mathcal{X}} \otimes M,
	\end{align*}
	where $e(\Phi)(s)=-\rho(s)(\Phi)$ is defined as above.
\end{prop}

The above proposition implies the following corollary:
\begin{cor}\label{302}
	If a Higgs bundle $(\mathcal{F},\Phi)$ is stable, then the tangent space of $\mathcal{M}_H(\mathcal{X})$ at the point $\xi=(\mathcal{F},\Phi)$ is isomorphic to $\mathbb{H}^1(C_{\varepsilon}^{\bullet})$, where $C_{\varepsilon}^{\bullet}$ is the complex
	\begin{align*}
	C_{\varepsilon}^{\bullet}:C_{\varepsilon}^0 =\mathcal{E}nd(\mathcal{F}) \longrightarrow C_{\varepsilon}^1 = \mathcal{E}nd(\mathcal{F}) \otimes \Omega^1_{\mathcal{X}}.
	\end{align*}
\end{cor}

Proposition 3.3 in \cite{Sun191} actually proves the statement for an $\mathcal{L}$-twisted Higgs bundle, and therefore, the result can be generalized to the $\mathcal{L}$-twisted case:
\begin{cor}\label{303}
	The tangent space of $\mathcal{M}_H(\mathcal{X},\mathcal{L})$ at a stable $\mathcal{L}$-twisted Higgs bundle $\xi=(\mathcal{F},\Phi)$ is isomorphic to $\mathbb{H}^1(C_{\varepsilon}^{\bullet})$, where $C_{\varepsilon}^{\bullet}$ is the complex
	\begin{align*}
	C_{\varepsilon}^{\bullet}:C_{\varepsilon}^0 =\mathcal{E}nd(\mathcal{F}) \longrightarrow C_{\varepsilon}^1 = \mathcal{E}nd(\mathcal{F}) \otimes \mathcal{L}.
	\end{align*}
\end{cor}

Let $\mathcal{E}$ be a principal $G$-bundle. We fix a faithful representation $G \hookrightarrow \text{GL}(V)$, and consider the associated bundle $\mathcal{E}(V)$. Then, we can use the same argument as in Proposition \ref{301} to calculate the tangent space of $\mathcal{M}_H(\mathcal{X},G)$.
\begin{cor}\label{304}
	The tangent space of $\mathcal{M}_H(\mathcal{X},G)$ at a stable $G$-Higgs bundle $\xi=(\mathcal{E},\Phi)$ is isomorphic to $\mathbb{H}^1(C_{\varepsilon,G}^{\bullet})$, where $C_{\varepsilon,G}^{\bullet}$ is the complex
	\begin{align*}
	C_{\varepsilon,G}^{\bullet}:C_{\varepsilon,G}^0 =\mathcal{E}(\mathfrak{g}) \rightarrow C_{\varepsilon,G}^1 = \mathcal{E}(\mathfrak{g}) \otimes \Omega^1_{\mathcal{X}}.
	\end{align*}
\end{cor}

\subsection{Grothendieck Duality}
Let $\mathcal{X}$ and $\mathcal{Y}$ be separated and finite type Deligne-Mumford stacks. Denote by $D^{\#}(\mathcal{X})$ the derived category of complexes of coherent sheaves over $\mathcal{X}$, where \# represents here either of ${\rm b,+,-}$.

Let $f: \mathcal{X} \rightarrow \mathcal{Y}$ be a proper morphism of stacks. The morphism induces the following functors of categories of coherent sheaves
\begin{align*}
f_*: \text{Coh}(\mathcal{X}) \rightarrow \text{Coh}(\mathcal{Y}), \quad f^*: \text{Coh}(\mathcal{Y}) \rightarrow \text{Coh}(\mathcal{X})
\end{align*}
such that $f_*$ is right adjoint to $f^*$. From these two functors, we can define the derived functors
\begin{align*}
Rf_*: D^b(\mathcal{X}) \rightarrow D^b(\mathcal{Y}), \quad Lf^*: D^b(\mathcal{Y}) \rightarrow D^b(\mathcal{X}).
\end{align*}
Note that the functor $Rf_*$ is still right adjoint to $Lf^*$:
\begin{align*}
{\rm Hom}(\mathcal{E},Rf_* \mathcal{F}) \cong {\rm Hom}(Lf^* \mathcal{E},\mathcal{F}).
\end{align*}
In fact, we have another functor $f^!: D^b(\mathcal{Y}) \rightarrow D^b(\mathcal{X})$, which is right adjoint to $Rf_*$, satisfying
\begin{align*}
Rf_* R\mathcal{H}om(\mathcal{E}^{\bullet},f^! \mathcal{F}^{\bullet}) \cong R\mathcal{H}om(Rf_* \mathcal{E}^{\bullet},\mathcal{F}^{\bullet}),
\end{align*}
where $\mathcal{E}^{\bullet} \in D^b(\mathcal{X})$ and $\mathcal{F}^{\bullet} \in D^b(\mathcal{Y})$.

\begin{thm}[Theorem 2.22 in \cite{Nir2}]\label{305}
	Let $\sigma: \mathcal{X} \rightarrow {\rm Spec} (\mathbb{C})$ be a smooth proper Deligne-Mumford stack of dimension $n$ over $\mathbb{C}$. Then $\sigma^! (\mathbb{C}_{\rm pt})$, where $\mathbb{C}_{\rm pt}$ is the constant sheaf on ${\rm Spec}(\mathbb{C})$, is canonically isomorphic to the complex $\omega_{\mathcal{X}}[n]$.
\end{thm}

Considering the morphism $f: \mathcal{X} \rightarrow \text{Spec} (\mathbb{C})$, we have
\begin{align*}
{\rm Hom}_{D^b(\mathcal{X})}(\mathcal{E}^{\bullet},\omega_{\mathcal{X}}[n]) \cong {\rm Hom}_k(R\Gamma(\mathcal{E}^{\bullet}),\mathbb{C}).
\end{align*}
If $\mathcal{E}^{\bullet}$ is a coherent sheaf $\mathcal{E}$ over $\mathcal{X}$, then ${\rm Ext}^i (\mathcal{E},\omega_{\mathcal{X}}) \cong H^{n-i}(\mathcal{X},\mathcal{E})^*$.

Let $\mathcal{E}^{\bullet}= 0 \rightarrow \mathcal{E}_0 \rightarrow \dots \rightarrow \mathcal{E}_m \rightarrow 0$ be an element in $D^b(\mathcal{X})$. By Grothendieck duality we have
\begin{align*}
H^{i}(\mathcal{X},\mathcal{E}^{\bullet}) \cong H^{1-i+n}(\mathcal{E}^{\bullet,*} \otimes \omega_{\mathcal{X}}),
\end{align*}
where $\mathcal{E}^{\bullet,*}$ is the ``dual" complex
\begin{align*}
0 \rightarrow (\mathcal{E}_m)^* \rightarrow \dots \rightarrow (\mathcal{E}_0)^* \rightarrow 0.
\end{align*}

\subsection{Application}
For a stacky curve $\mathcal{X}$, we then have $\Omega^1_{\mathcal{X}} \cong \omega_{\mathcal{X}}$. In \S 3.2, we have seen that the tangent space $T_{\xi}(\mathcal{M}_H(\mathcal{X}))$ of the moduli space $\mathcal{M}_H(\mathcal{X})$ at the point $\xi=(\mathcal{F},\Phi)$ is isomorphic to $\mathbb{H}^1(C_{\varepsilon}^{\bullet})$, where \begin{align*}
C_{\varepsilon}^{\bullet}:C_{\varepsilon}^0 =\mathcal{E}nd(\mathcal{F}) \longrightarrow C_{\varepsilon}^1 = \mathcal{E}nd(\mathcal{F}) \otimes \omega_{\mathcal{X}}
\end{align*}
is a 2-term complex. Thus the cotangent space $T^*_{\xi}(\mathcal{M}_H(\mathcal{X}))$ of the moduli space is isomorphic to $\mathbb{H}^1(C_{\varepsilon}^{\bullet,*})$, where $C_{\varepsilon}^{\bullet,*}$ is the dual of the complex $C_{\varepsilon}^{\bullet}$. Tensoring the complex $C_{\varepsilon}^{\bullet,*}$ by $\omega_{\mathcal{X}}$, the complex $C_{\varepsilon}^{\bullet}$ is dual to $C_{\varepsilon}^{\bullet,*} \otimes \omega_{\mathcal{X}}$ in the derived category by Grothendieck duality. Note that there is a natural morphism $C_{\varepsilon}^{\bullet,*} \rightarrow C_{\varepsilon}^{\bullet,*}\otimes \omega_{\mathcal{X}}$. This induces the morphism
\begin{align*}
C_{\varepsilon}^{\bullet,*}\rightarrow C_{\varepsilon}^{\bullet,*}\otimes \omega_{\mathcal{X}} \xrightarrow{\cong} C_{\varepsilon}^{\bullet}
\end{align*}
and so 
\begin{equation}\label{eq_bivector}
T^*_{\xi}(\mathcal{M}_H(\mathcal{X})) \cong \mathbb{H}^{1}(C_{\varepsilon}^{\bullet,*}) \longrightarrow \mathbb{H}^{1}(C_{\varepsilon}^{\bullet}) \cong T_{\xi}(\mathcal{M}_H(\mathcal{X})).
\end{equation}
Similarly, we also have the morphism $T^*_{\xi}(\mathcal{M}_H(\mathcal{X},G)) \rightarrow T_{\xi}(\mathcal{M}_H(\mathcal{X},G))$ induced by the morphism $C_{\varepsilon,G}^{\bullet,*}\rightarrow  C_{\varepsilon,G}^{\bullet}$. The significance of this map will be demonstrated in the next section, where we will show that this induces a Poisson structure on the moduli spaces $\mathcal{M}_H(\mathcal{X},\alpha)$ and $\mathcal{M}_H(\mathcal{X},G,\alpha)$.

\section{Lie Algebroids and Poisson Structures}

The principal method by which we shall obtain Poisson structures on our moduli spaces of interest is via the duals of Lie algebroids. Poisson structures are usually defined on smooth varieties. Nonetheless, Poisson structures on stacks were first conceived in the dissertation of Waldron \cite{Wal}. Although the author only considers differential stacks in \cite{Wal}, the approach can be carried over to the algebraic setting naturally. In \S 4.1, we shall review the definitions and some properties of Lie algebroids and Poisson structures on smooth varieties, as well as on stacks;  we refer to \cite[\S 7.2]{Wal} and \cite[\S 1]{Kaledin} for a complete overview. In \S 4.2 and \S 4.3, we prove the main theorems of this section that the moduli spaces $\mathcal{M}_H(\mathcal{X},\alpha)$ and $\mathcal{M}_H(\mathcal{X},G,\alpha)$ admit a Poisson structure (Theorems \ref {404} and \ref{405}). In \S 4.4, we consider the stack $\widetilde{\mathcal{M}}_H(\mathcal{X},\alpha)$ of Higgs bundles on $\mathcal{X}$ by abuse of notation. As an application of Theorem \ref{404}, we next show that $\widetilde{\mathcal{M}}_H(\mathcal{X},\alpha)$ also admits a Poisson structure.

\subsection{Lie algebroids and Poisson structures}
We start by recalling the basic notions for Lie algebroids and Poisson structures on smooth varieties over $\mathbb{C}$ and then pass to a reasonable generalisation of these notions over stacks.

\subsubsection*{\textbf{Smooth Varieties}}
A \emph{Lie algebroid} on a smooth variety $X$ is a vector bundle $F\to X$ together with a Lie bracket $\left[ \text{ } , \text{ } \right]$ on the space of global sections $\Gamma ( F )$ and a morphism $a: F \to TX$, called the \emph{anchor map} of $F$, which induces a Lie algebra morphism $\Gamma( F )\to \Gamma \left( TX \right)$ satisfying the Leibniz rule
\[\left[ \xi ,f\nu  \right]=a(\xi)(f)\nu +f\left[ \xi ,\nu  \right],\]
for all $\xi ,\nu \in \Gamma ( F )$ and $f\in {{C}^{\infty }}\left( X \right)$. A \emph{morphism} of Lie algebroids on $X$ is a morphism of vector bundles inducing a Lie algebra morphism between spaces of sections and commuting with the anchor maps. The category of Lie algebroids on $X$ is denoted by $\mathsf{\mathcal{L}\mathcal{A}}$.

A \emph{Poisson bracket} on $X$ is a Lie bracket 
\begin{align*}
\left\{ \text{ } , \text{ } \right\}: \mathcal{O}_X \times \mathcal{O}_X \rightarrow \mathcal{O}_X
\end{align*}
satisfying the Leibniz rule $\left\{ f,gh \right\}=\left\{ f,g \right\}h+g\left\{ f,h \right\}$, for  $f,g,h \in \mathcal{O}_X$ (cf. \cite[Section 1]{Kaledin}). Poisson brackets bijectively correspond to bi-vector fields
\[\Pi \in \Gamma \left( {{\wedge }^{2}}TX \right)\]
such that
\[\left[ \Pi ,\Pi  \right]=0.\]
The Poisson bracket ${{\left\{ , \right\}}_{\Pi }}$ that corresponds to such a bi-vector field $\Pi $ is given by ${{\left\{ f,g \right\}}_{\Pi }}=\Pi \left( df,dg \right)$. Associated to a Poisson structure $\Pi$ on $X$, there is a Lie algebroid structure $T^*_{\Pi}X$ on $T^*X$ (see \cite[\S 7.2.1]{Wal}). The class of examples of Poisson manifolds that we are interested in is the one that arises from Lie algebroids through the following theorem, which in the case of differentiable manifolds is due to Courant \cite[Theorem 2.1.4]{Cour}:

\begin{thm}\label{401}
	If $F$ is a Lie algebroid on $X$, then the total space of the dual vector bundle $F^{*}$ has a natural Poisson structure.
\end{thm}

We now consider an important example of Lie algebroids, the \emph{Atiyah algebroid}. Let $X \to X/H$ be an $H$-torsor. Denote by $\mathfrak{h}$ the corresponding Lie algebra of $H$. We have a natural projection $\pi: X \rightarrow X/H$ and, moreover, we obtain the following exact sequence
\begin{align*}
0 \rightarrow T_{\rm orbits} X \rightarrow TX \rightarrow \pi^* T(X/H) \rightarrow 0.
\end{align*}
The group action $H$ on $X$ induces a natural action on $TX$ and a natural surjective morphism $TX/H \rightarrow T(X/H)$, which is regarded as an anchor map. Therefore, we have the following exact sequence
\begin{align*}
0 \rightarrow {\rm Ad}(X) \rightarrow TX/H \xrightarrow{a} T(X/H) \rightarrow 0,
\end{align*}
where ${\rm Ad}(X):=X \times_{\rm Ad}\mathfrak{h}$. This exact sequence is called the \emph{Atiyah sequence}. It is easy to check that the Atiyah sequence gives a Lie algebroid structure on $TX/H$. By Theorem \ref{401}, the total space of $(TX/H)^*$ has a Poisson structure; note here that we are viewing $TX/H$ as a vector bundle over $X/H$.

\subsubsection*{\textbf{Stacks}}
We will extend the definition of Lie algebroids from schemes to stacks. Let $\mathcal{M}$ be a smooth Deligne-Mumford stack. Take a surjective \'etale morphism $M \rightarrow \mathcal{M}$, where $M$ is a smooth variety. We consider the simplicial resolution of $\mathcal{M}$
\begin{align*}
M \times_{\mathcal{M}} M \overset{s}{\underset{t}\rightrightarrows}M \rightarrow \mathcal{M}.
\end{align*}
Similar to the definition of sheaves on stacks (see \S2.2), we still use the local chart $M \rightarrow \mathcal{M}$ to give the definition of Lie algebroids on $\mathcal{M}$. A \emph{Lie algebroid $\mathcal{F}$} on $\mathcal{M}$ is defined as a pair $(F,\sigma)$, where $F$ is a Lie algebroid on $M$ and $\sigma: s^* F \xrightarrow{\cong} t^*F$ is an isomorphism of Lie algebroids on $M \times_{\mathcal{M}} M$. We refer the reader to \cite[\S 3.4]{Wal} for the definition of pullbacks of Lie algebroids. The definition does not depend on the choice of charts of $\mathcal{M}$.

Now we move to Poisson structures on $\mathcal{M}$. A \emph{Poisson structure} on $\mathcal{M}$ is a pair $(\Pi,\sigma)$, where $\Pi$ is a Poisson structure on $M$ and $\sigma: s^* \Pi \xrightarrow{\cong} t^* \Pi$ is an isomorphism of Poisson structures on $M \times_{\mathcal{M}} M$. With the same approach as for smooth varieties, we have the desired generalization of Theorem \ref{401}:
\begin{prop}{\cite[\S 7.2.5]{Wal}}\label{403}
	Let $\mathcal{F}$ be a Lie algebroid over $\mathcal{M}$. The total space of $\mathcal{F}^*$ has a natural Poisson structure.
\end{prop}

We will see in \S 4.4 that the moduli problems we are studying in this article fall in the case described by the above proposition.

\subsection{Poisson Structure on $\mathcal{M}_H(\mathcal{X}, \alpha )$}
Let $\mathcal{X}=[Y/{\Gamma}]$ be a stacky curve over $\mathbb{C}$ and let $X$ be the coarse moduli space of $\mathcal{X}$. In this section, we will prove the main theorem of this paper.
\begin{thm}\label{404}
	The moduli space $\mathcal{M}_H (\mathcal{X}, \alpha)$ of stable Higgs bundles over $\mathcal{X}$ with fixed parabolic structure $\alpha$ admits a Poisson structure.
\end{thm}

\begin{proof}
The strategy of the proof is to show that $\mathcal{M}_H(\mathcal{X},\alpha)$ is a Lie algebroid over $T\mathcal{M}(X)$, the tangent space of the moduli space of stable bundles over $X$, by constructing a map $$\mathcal{M}_H(\mathcal{X},\alpha) \rightarrow T\mathcal{M}(X),$$
which will play the role of an anchor map. Note that \eqref{eq_bivector} gives precisely such a map. Following the approach as in \cite[\S 3]{LoMa}, we first restrict to $\mathcal{M}^0_{H}(\mathcal{X},\alpha)$, the moduli space of stable Higgs bundles $(\mathcal{F},\Phi)$ where the underlying bundle $\mathcal{F}$ is a stable bundle. Since this moduli space is a dense open subset of $\mathcal{M}_H(\mathcal{X},\alpha)$, if there is a Poisson structure on $\mathcal{M}^0_{H}(\mathcal{X},\alpha)$ with the bi-vector field induced by \eqref{eq_bivector}, then the anchor map is used to give a Poisson structure on $\mathcal{M}_H(\mathcal{X},\alpha)$. In the sequel, we will construct a Poisson structure on $\mathcal{M}^0_{H}(\mathcal{X},\alpha)$, of which the associated bi-vector field is given by \eqref{eq_bivector}.
	
By the discussion in \S 2.4, let $\mathcal{X}$ be a root stack $X_{\bar{D},\bar{r}}$. Denote by $\pi: \mathcal{X} \rightarrow X$ the natural morphism; note that the dimension of $\mathcal{X}$ is one. The divisor $\bar{D}=p_1 + \dots + p_n$ is a sum of distinct points and let $q_i$ be the corresponding point of $p_i$ in $\mathcal{X}$. Note that $\pi^{-1}(p_i)=r(p_i)q_i$. Let us denote by $\mathbb{D}=q_1+ \dots + q_n$ the divisor in $\mathcal{X}$.
	
Let $\mathcal{F}$ be a locally free sheaf on $\mathcal{X}$ with parabolic structure $\alpha$. Given $q \in \mathbb{D}$, denote by $\alpha(q)$ the parabolic structure of $\pi_*\mathcal{F}$ around $p=\pi(q) \in \bar{D}$. Let $P_q$ be the parabolic group of $\text{GL}_n(\mathbb{C})$ corresponding to $\alpha(q)$ as we discussed in \S 2.5. Let $P_q=L_q N_q$ be the Levi decomposition of $P$, where $L_q$ is the Levi factor and $N_q$ is a unipotent group. Denote by $\mathfrak{l}_q,\mathfrak{n}_q$ the Lie algebras of $L_q$ and $N_q$ respectively. Define $\mathcal{F}':=\pi^* \pi_* \mathcal{F}$. Since the stacky curve $\mathcal{X}$ is defined over $\mathbb{C}$, the functor $\pi_*$ is exact. Thus, $\mathcal{F}'$ is a locally free sheaf on $\mathcal{X}$. We have the following exact sequence
	\begin{equation}\tag{4.1}\label{eq41}
	0 \rightarrow \mathcal{E}nd(\mathcal{F}) \rightarrow \mathcal{E}nd(\mathcal{F}') \rightarrow \prod_{q \in \mathbb{D}} \mathfrak{n}_q \otimes \mathcal{O}_{q} \rightarrow 0.
	\end{equation}
	This induces a long exact sequence
	\begin{align*}
	0 & \rightarrow  {\rm End}(\mathcal{F}) \rightarrow {\rm End}(\mathcal{F}') \rightarrow H^0(\mathcal{X}, \prod_{q \in \mathbb{D}} \mathfrak{n}_q \otimes \mathcal{O}_{q}) \rightarrow \\
	& \rightarrow {\rm Ext}^1(\mathcal{F},\mathcal{F})  \rightarrow {\rm Ext}^1(\mathcal{F}',\mathcal{F}') \rightarrow H^1(\mathcal{X}, \prod_{q \in \mathbb{D}} \mathfrak{n}_q \otimes \mathcal{O}_{q}) \rightarrow 0.
	\end{align*}
	Note that the last term $H^1(\mathcal{X}, \prod_{q \in \mathbb{D}} \mathfrak{n}_q \otimes \mathcal{O}_{q})$ is trivial, and thus we have a short exact sequence
	\begin{equation}\tag{4.2}\label{eq421}
	0 \rightarrow {\rm Ad} \rightarrow {\rm Ext}^1(\mathcal{F},\mathcal{F})  \rightarrow {\rm Ext}^1(\mathcal{F}',\mathcal{F}') \rightarrow 0,
	\end{equation}
	where ${\rm Ad}$ is the kernel of the surjective map ${\rm Ext}^1(\mathcal{F},\mathcal{F})  \rightarrow {\rm Ext}^1(\mathcal{F}',\mathcal{F}')$. Remember that $\mathcal{F}$ is a stable bundle and we are working over the field $\mathbb{C}$. Therefore, $\mathcal{F}$ is simple and the functor
	\begin{align*}
	\pi_*: \text{Coh}(\mathcal{X}) \rightarrow \text{Coh}(X)
	\end{align*}
	is exact. The exactness of the functor $\pi_*$ implies that $\pi_*\mathcal{F}$ is stable, and so is $\mathcal{F}'$. Therefore, ${\rm End}(\mathcal{F}')\cong \mathbb{C}$. Now we go back to the term ${\rm Ad}$. If $\mathcal{F}$ is stable, then
	\begin{align*}
	{\rm Ad} \cong H^0(\mathcal{X}, \prod_{q \in \mathbb{D}} \mathfrak{n}_q \otimes \mathcal{O}_{q}),
	\end{align*}
	which is supported over $q \in \mathbb{D}$. Over each point $q \in \mathbb{D}$, ${\rm Ad}_q$ is isomorphic to the Lie algebra $\mathfrak{n}_q$. Therefore, ${\rm Ad}_q$ is the adjoint representation of the unipotent group $N_q$.
	
	In the short exact sequence \eqref{eq421}, the third term ${\rm Ext}^1(\mathcal{F}',\mathcal{F}')$ is isomorphic to the tangent space of the moduli space of stable bundles over $X$ at the point $\mathcal{F}'$, thus
	\begin{align*}
	{\rm Ext}^1(\mathcal{F}',\mathcal{F}') \cong T_{\mathcal{F}'}(\mathcal{M}(X)).
	\end{align*}
	With respect to this isomorphism, we have a natural map
	\begin{align*}
	{\rm Ext}^1(\mathcal{F},\mathcal{F})  \rightarrow T_{\mathcal{F}'}(\mathcal{M}(X)).
	\end{align*}
	Now we consider the tangent bundles on $\mathcal{M}(\mathcal{X},\alpha)$ and $\mathcal{M}(X)$
	\begin{center}
		\begin{tikzcd}
		&  \mathcal{M}(\mathcal{X},\alpha) \times \mathcal{X} \arrow[dl,swap, "\mu_1"] \arrow[dr, "\nu_1"] & & &  \mathcal{M}(X) \times X \arrow[dl,swap, "\mu_2"] \arrow[dr, "\nu_2"] & \\
		\mathcal{M}(\mathcal{X},\alpha)   & & \mathcal{X} & \mathcal{M}(X)   & & X.
		\end{tikzcd}
	\end{center}
	Let $\mathscr{F}$, $\mathscr{F}'$ be the universal bundles on $\mathcal{M}(\mathcal{X},\alpha) \times \mathcal{X}$ and $\mathcal{M}(X) \times X$ respectively. Therefore, we have
	\begin{align*}
	0 \rightarrow \mathcal{E}nd(\mathscr{F}) \rightarrow \mathcal{E}nd(\mathscr{F}') \rightarrow \prod_{q \in \mathbb{D}} \mathfrak{n}_q \otimes \mathcal{O}_{\nu_1^{-1}(q)} \rightarrow 0
	\end{align*}
	and
	\begin{align*}
	0 \rightarrow \mathscr{A}d \rightarrow R^1 (\mu_1)_* \mathcal{E}nd(\mathscr{F})  \rightarrow R^1 (\mu_1)_* \mathcal{E}nd(\mathscr{F}') \rightarrow 0.
	\end{align*}
	Clearly, the term
	\begin{align*}
	R^1 (\mu_1)_* \mathcal{E}nd(\mathscr{F}) \cong T\mathcal{M}(\mathcal{X},\alpha)
	\end{align*}
	is the tangent bundle of $\mathcal{M}(\mathcal{X},\alpha)$, and
	\begin{align*}
	R^1 (\mu_1)_* \mathcal{E}nd(\mathscr{F}') \cong T\mathcal{M}(X)
	\end{align*}
	is the tangent bundle of $\mathcal{M}(X)$. 
	
	For the rest of the proof, we prove the following two statements
	\begin{enumerate}
		\item $(R^1 (\mu_1)_* \mathcal{E}nd(\mathscr{F}))^*$ is isomorphic to $\mathcal{M}^0_H(\mathcal{X},\alpha)$ as bundles over $\mathcal{M}(\mathcal{X},\alpha)$, and
		\item $R^1 (\mu_1)_* \mathcal{E}nd(\mathscr{F})$ is a Lie algebroid over $\mathcal{M}(X)$,
	\end{enumerate}
	which shall imply that the moduli space $\mathcal{M}^{0}_H(\mathcal{X},\alpha)$ admits a Poisson structure by Theorem \ref{401}.
	
	To prove the first statement, we work locally on a point $\mathcal{F} \in \mathcal{M}(\mathcal{X},\alpha)$. The space ${\rm Ext}^1(\mathcal{F},\mathcal{F})$ is the tangent space of $\mathcal{M}(\mathcal{X},\alpha)$ at the point $\mathcal{F}$, and we have
	\begin{align*}
	H^1(\mathcal{X},\mathcal{E}nd(\mathcal{F}) ) \cong {\rm Ext}^1(\mathcal{F},\mathcal{F}) \cong T_{ \mathcal{F} }(\mathcal{M}(\mathcal{X},\alpha)).
	\end{align*}
	By Grothendieck duality (see \S 3.3), the following isomorphism holds
	\begin{align*}
	H^1(\mathcal{X},\mathcal{E}nd(\mathcal{F}) ) \cong H^0(\mathcal{X},\mathcal{E}nd(\mathcal{F}) \otimes \omega_{\mathcal{X}})^*.
	\end{align*}
	This isomorphism tells us that
	\begin{align*}
	T^*_{ \mathcal{F} }(\mathcal{M}(\mathcal{X},\alpha)) \cong H^0(\mathcal{X},\mathcal{E}nd(\mathcal{F}) \otimes \omega_{\mathcal{X}}),
	\end{align*}
	and let $\Phi \in H^0(\mathcal{X},\mathcal{E}nd(\mathcal{F}) \otimes \omega_{\mathcal{X}})$ be a Higgs field. This finishes the proof of the first statement. Furthermore, we have
	\begin{align*}
	T_{\Phi} T^*_{\mathcal{F}}(\mathcal{M}(\mathcal{X},\alpha)) \cong \mathbb{H}^1(\mathcal{E}nd(\mathcal{F}) \rightarrow \mathcal{E}nd(\mathcal{F}) \otimes \omega_{\mathcal{X}}) \cong T_{(\mathcal{F},\Phi)}(\mathcal{M}^0_H(\mathcal{X},\alpha)).
	\end{align*}
	For the second statement, note that in summary we have the following exact sequence
	\begin{align*}
	0 \rightarrow \mathscr{A}d \rightarrow T\mathcal{M}(\mathcal{X},\alpha)  \rightarrow T\mathcal{M}(X) \rightarrow 0,
	\end{align*}
	which is an Atiyah sequence as studied in \S 4.1. Therefore, the second statement also holds and this finishes the proof of the theorem.
\end{proof}

\subsection{Poisson Structure on $\mathcal{M}_H(\mathcal{X},G,\alpha)$}
Let $G \hookrightarrow \text{GL}(V)$ be a fixed faithful representation. In this subsection, we will prove that there exists a Poisson structure on $\mathcal{M}_H(\mathcal{X},G,\alpha)$.
\begin{thm}\label{405}
	The moduli space $\mathcal{M}_H(\mathcal{X},G,\alpha)$ of stable  $G$-Higgs bundles over a stacky curve $\mathcal{X}$  with fixed parabolic structure $\alpha$ admits a Poisson structure.
\end{thm}

\begin{proof}
	The proof is similar to the one for the moduli space $\mathcal{M}_H(\mathcal{X},\alpha)$. We work on the open dense subset $\mathcal{M}^0_H(\mathcal{X},G,\alpha)$, where the underlying principal $G$-bundles are also stable. Denote by $\mathcal{M}(X,G)$ the moduli space of stable principal $G$-bundles on $X$ and consider the tangent bundles
	\begin{center}
		\begin{tikzcd}
		&  \mathcal{M}(\mathcal{X},G,\alpha) \times \mathcal{X} \arrow[dl,swap, "\mu_1"] \arrow[dr, "\nu_1"] & & &  \mathcal{M}(X,G) \times X \arrow[dl,swap, "\mu_2"] \arrow[dr, "\nu_2"] & \\
		\mathcal{M}(\mathcal{X},G,\alpha)   & & \mathcal{X} & \mathcal{M}(X,G)   & & X.
		\end{tikzcd}
	\end{center}
	Let $\mathscr{E}$, $\mathscr{E}'$ be the universal bundles on $\mathcal{M}(\mathcal{X},G,\alpha) \times \mathcal{X}$ and $\mathcal{M}(X,G) \times X$ respectively. Denote by $\mathscr{E}(V)$ and $\mathscr{E}'(V)$ the associated bundles with respect to $G \hookrightarrow \text{GL}(V)$. Around a point $q \in \mathbb{D}$, let $P_q$ be the corresponding parabolic group in $\text{GL}(V)$ with respect to $\mathscr{E}(V)$. Denote by $P_q=L_q N_q$ the Levi factorization. With the same discussion as in the proof of Theorem \ref{404}, we have
	\begin{align*}
	0 \rightarrow \mathcal{E}nd(\mathscr{E}(V)) \rightarrow \mathcal{E}nd(\mathscr{E}'(V)) \rightarrow \prod_{q \in \mathbb{D}} \mathfrak{n}_q \otimes \mathcal{O}_{\nu_1^{-1}(q)} \rightarrow 0.
	\end{align*}
	Let $P'_q$, $L'_q$ and $N'_q$ be the pre-images of $P_q$, $L_q$ and $N_q$ in $G$ respectively  via the fixed faithful representation. Therefore, we have
	\begin{align*}
	0 \rightarrow \mathcal{E}nd(\mathscr{E}) \rightarrow \mathcal{E}nd(\mathscr{E}') \rightarrow \prod_{q \in \mathbb{D}} \mathfrak{n}'_q \otimes \mathcal{O}_{\nu_1^{-1}(q)} \rightarrow 0,
	\end{align*}
	where $\mathfrak{n}'_q$ is the Lie algebra of $N'_q$. This short exact sequence induces the following one
	\begin{align*}
	0 \rightarrow \mathscr{A}d \rightarrow R^1 (\mu_1)_* \mathcal{E}nd(\mathscr{E})  \rightarrow R^1 (\mu_1)_* \mathcal{E}nd(\mathscr{E}') \rightarrow 0.
	\end{align*}
	Clearly, we have
	\begin{align*}
	R^1 (\mu_1)_* \mathcal{E}nd(\mathscr{E}) \cong T\mathcal{M}(\mathcal{X},G,\alpha)
	\end{align*}
	and
	\begin{align*}
	R^1 (\mu_1)_* \mathcal{E}nd(\mathscr{E}') \cong T\mathcal{M}(X,G),
	\end{align*}
	the tangent bundle of $\mathcal{M}(X,G)$. Therefore, $R^1 (\mu_1)_* \mathcal{E}nd(\mathscr{E})$ is a Lie algebroid over $\mathcal{M}(X,G)$. At the same time, the tangent space of $(R^1 (\mu_1)_* \mathcal{E}nd(\mathscr{E}))^*$ is isomorphic to the tangent space of $\mathcal{M}^0_H(\mathcal{X},G,\alpha)$, in other words,
	\begin{align*}
	T(R^1 (\mu_1)_* \mathcal{E}nd(\mathscr{E}))^* \cong T\mathcal{M}^0_H(\mathcal{X},G,\alpha).
	\end{align*}
	Therefore, the moduli space of $G$-Higgs bundles $\mathcal{M}_H(\mathcal{X},G,\alpha)$ has a Poisson structure.
\end{proof}


\subsection{Stacks}

In the last subsection, we proved that there exists a Poisson structure on the moduli space $\mathcal{M}_H(\mathcal{X},\alpha)$. By abuse of notation, we consider the existence of a Poisson structure on the moduli stack $\widetilde{\mathcal{M}}_H(\mathcal{X},\alpha)$ of Higgs bundles on $\mathcal{X}$. We show that the stack $\widetilde{\mathcal{M}}_H(\mathcal{X},\alpha)$ has a Poisson structure, which is induced by the Poisson structure on $\mathcal{M}_H(\mathcal{X},\alpha)$. Below we sketch the extension of the construction of a Poisson structure on a general stack $\widetilde{\mathcal{M}}$.

Let $\mathcal{X}$ be a stacky curve with coarse moduli space $X$. It is known that $\widetilde{\mathcal{M}}_H(\mathcal{X},\alpha)$ has a natural stack structure \cite{CasaWise}. Since this moduli problem is defined for the stable Higgs bundles, $\mathcal{M}_H(\mathcal{X},\alpha)$ is a fine moduli space of $\widetilde{\mathcal{M}}_H(\mathcal{X},\alpha)$ \cite[Theorem 1.3]{Sun201}. Thus,
\begin{align*}
\widetilde{\mathcal{M}}_H(\mathcal{X},\alpha)(-) \cong {\rm Hom}(-, \mathcal{M}_H(\mathcal{X},\alpha)).
\end{align*}
This isomorphism provides the following corollary:
\begin{cor}\label{408}
	The stack $\widetilde{\mathcal{M}}_H(\mathcal{X},\alpha)$ admits a Poisson structure.
\end{cor}

In general, a moduli problem may not have a fine moduli space, for example $\widetilde{\mathcal{M}}^{ss}_H(\mathcal{X},\alpha)$ the moduli problem of semistable Higgs bundles. To deal with the general case, we have to find another approach to construct a Poisson structure on the corresponding stack structure.

In the rest of this subsection, we give a brief idea about the construction of a Poisson structure on a stack $\widetilde{\mathcal{M}}$. The idea is that we want to find an atlas $\{M_i\}_{i \in \mathscr{I}}$ of $\widetilde{\mathcal{M}}$ in the Lisse-\'etale site (see \cite[Example 2.1.15]{Ol}), where $M_i$ are schemes. We can try to construct a Poisson structure on each $M_i$ and check whether they can be glued together. Furthermore, if $\widetilde{\mathcal{M}}$ is an algebraic stack, we can assume that there exists a smooth surjective morphism $M \rightarrow \widetilde{\mathcal{M}}$, where $M$ is a scheme, and construct a Poisson structure on $M$. We also have to check that the pull-backs $s,t: M \times_{\widetilde{\mathcal{M}}} M \rightarrow M$ of the Poisson structure on $M$ are isomorphic.

Now we take the stack $\widetilde{\mathcal{M}}^{ss}_H(\mathcal{X},\alpha)$ as an example. Let $\xi=(\mathcal{F},\Phi)$ be a point in $\widetilde{\mathcal{M}}^{ss}_H(\mathcal{X},\alpha)$. There is a quasi-projective substack $\widetilde{M}_{\xi} \subseteq \widetilde{\text{Quot}}(\mathcal{G},P)$, where $\mathcal{G}$ is a coherent sheaf, $P$ is a (Hilbert) polynomial and $\widetilde{\text{Quot}}$ is the \emph{Quot-functor} (see \cite[\S 1]{OlSt}), such that $\widetilde{M}_{\xi} \rightarrow \widetilde{\mathcal{M}}^{ss}_H(\mathcal{X},\alpha)$ is a smooth morphism (see the proof of Proposition 6.3 in \cite{Sun201}). Note that the Quot-functor $\widetilde{\text{Quot}}(\mathcal{G},P)$ is represented by a quasi-projective scheme $\text{Quot}(\mathcal{G},P)$ \cite[Theorem 4.4]{OlSt}. Denote by $M_{\xi}$ the subscheme of $\text{Quot}(\mathcal{G},P)$ representing $\widetilde{M}_{\xi}$. Then, we have a smooth morphism from a scheme $M_{\xi}$ to $\widetilde{\mathcal{M}}_H(\mathcal{X},\alpha)$. Running over all points in $\widetilde{\mathcal{M}}^{ss}_H(\mathcal{X},\alpha)$, we get an atlas $\{M_{\xi}\}_{\xi \in \widetilde{\mathcal{M}}^{ss}_H(\mathcal{X},\alpha)}$ of $\widetilde{\mathcal{M}}^{ss}_H(\mathcal{X},\alpha)$. With respect to the atlas we find, we can work on $M_{\xi}$ and try to construct a Poisson structure on it. The scheme $M_{\xi}$ is a subscheme in $\text{Quot}(\mathcal{G},P)$. The problem now is that we have to glue the Poisson structures on each local chart $M_{\xi}$ together and construct the Poisson structure globally. We thus conjecture:

\begin{conj}
	There is a Poisson structure on the stack $\widetilde{\mathcal{M}}^{ss}_H(\mathcal{X},\alpha)$ of semistable Higgs bundles over $\mathcal{X}$ with fixed parabolic structure $\alpha$.
\end{conj}

\section{Poisson Structure on the Moduli Space of stable $\mathcal{L}$-twisted Higgs Bundles over Stacky Curves}
In this section, we work on $\mathcal{M}_H(\mathcal{X},\mathcal{L},\alpha)$, the moduli space of stable $\mathcal{L}$-twisted Higgs bundles on a stacky curve $\mathcal{X}$ with fixed parabolic structure $\alpha$. We prove that a Poisson structure exists over $\mathcal{M}_H(\mathcal{X},\mathcal{L},\alpha)$ under  certain conditions.

Let $\mathcal{X}=[U/\Gamma]$ be a stacky curve. Denote by $X$ the coarse moduli space of $\mathcal{X}$. Let $\alpha$ be a parabolic structure on $X$, and denote by $\bar{D}=p_1 + \dots+p_k \in X$ the divisor with respect to $\alpha$. Let $\mathbb{D}=q_1 +\dots + q_k \in \mathcal{X}$ be the corresponding divisor on $\mathcal{X}$, where $q_i$ is the point corresponding to $p_i$.

\begin{thm}\label{501}
	If there exists a short exact sequence
	\begin{equation}\tag{5.1}\label{eq501}
	0 \rightarrow  \mathcal{H}om(\mathcal{F}\otimes \mathcal{L},\mathcal{F}  \otimes \omega_{\mathcal{X}} ) \rightarrow \mathcal{E}nd(\mathcal{F}) \rightarrow \mathfrak{n} \rightarrow 0,
	\end{equation}
	for any stable bundle $\mathcal{F} \in \mathcal{M}(\mathcal{X},\alpha)$  such that
	\begin{enumerate}
		\item the morphism $\mathcal{H}om(\mathcal{F}\otimes \mathcal{L},\mathcal{F}  \otimes \omega_{\mathcal{X}} )\rightarrow \mathcal{E}nd(\mathcal{F})$ is not surjective,
		\item $\mathfrak{n}$ is a sheaf of Lie algebras supported on $\mathbb{D}$,
	\end{enumerate}
	then the moduli space $\mathcal{M}_H(\mathcal{X},\mathcal{L},\alpha)$ admits a Poisson structure.
\end{thm}

\begin{proof}
	Analogously to the proof of Theorem \ref{404}, we only have to work with $\mathcal{M}^0_H(\mathcal{X},\mathcal{L},\alpha)$, the moduli space of $\mathcal{L}$-twisted stable Higgs bundles over $\mathcal{X}$, such that the underlying locally free sheaf is stable.\\
	If the short exact sequence \eqref{eq501} exists, we then obtain a long exact sequence
	\begin{align*}
	0 & \rightarrow  {\rm Hom}(\mathcal{F}\otimes \mathcal{L},\mathcal{F}  \otimes \omega_{\mathcal{X}} ) \rightarrow {\rm End}(\mathcal{F}) \rightarrow H^0(\mathcal{X}, \mathfrak{n}) \rightarrow  \\
	& \rightarrow {\rm Ext}^1(\mathcal{F}\otimes \mathcal{L},\mathcal{F}  \otimes \omega_{\mathcal{X}} )  \rightarrow {\rm Ext}^1(\mathcal{F},\mathcal{F}) \rightarrow H^1(\mathcal{X}, \mathfrak{n}) \rightarrow 0.\label{Ltwist}
	\end{align*}
	By our assumption that the support of $\mathfrak{n}$ is contained in $\mathbb{D}$, the last term $H^1(\mathcal{X}, \mathfrak{n})$ is trivial. This implies a short exact sequence
	\begin{equation}\tag{5.2}\label{eq502}
	0 \rightarrow {\rm Ad} \rightarrow {\rm Ext}^1(\mathcal{F} \otimes \mathcal{L},\mathcal{F}  \otimes \omega_{\mathcal{X}})  \rightarrow {\rm Ext}^1(\mathcal{F},\mathcal{F}) \rightarrow 0,
	\end{equation}
	where ${\rm Ad}$ is the kernel of the map ${\rm Ext}^1(\mathcal{F}\otimes \mathcal{L} ,\mathcal{F}  \otimes \omega_{\mathcal{X}})  \rightarrow {\rm Ext}^1(\mathcal{F},\mathcal{F})$. Let $\mathscr{F}$ be the universal bundle on $\mathcal{M}(\mathcal{X},\alpha) \times \mathcal{X}$ and consider
	\begin{center}
		\begin{tikzcd}
		&  \mathcal{M}(\mathcal{X},\alpha) \times \mathcal{X} \arrow[dl,swap, "\mu_1"] \arrow[dr, "\nu_1"] & \\
		\mathcal{M}(\mathcal{X},\alpha)   & & \mathcal{X}
		\end{tikzcd}
	\end{center}
	The short exact sequence \eqref{eq502} induces the following sequence for universal bundles
	\begin{equation}\tag{5.3}\label{eq503}
	0 \rightarrow \mathscr{A}d \rightarrow R^1 (\mu_1)_* \mathcal{H}om(\mathscr{F} \otimes \nu_1^*\mathcal{L},\mathscr{F}\otimes \nu_1^*\omega_{\mathcal{X}})  \rightarrow R^1 (\mu_1)_* \mathcal{E}nd(\mathscr{F}) \rightarrow 0.
	\end{equation}
	With respect to our assumption about $\mathfrak{n}$, the short exact sequence \eqref{eq503} gives a Lie algebroid structure on $R^1 (\mu_1)_* \mathcal{H}om(\mathscr{F} \otimes \nu_1^*\mathcal{L},\mathscr{F}\otimes \nu_1^*\omega_{\mathcal{X}})$. Note that
	\begin{align*}
	R^1 (\mu_1)_* \mathcal{E}nd(\mathscr{F}) \cong T \mathcal{M}(\mathcal{X},\alpha).
	\end{align*}
	Therefore, $R^1 (\mu_1)_* \mathcal{H}om(\mathscr{F} \otimes \nu_1^*\mathcal{L},\mathscr{F}\otimes \nu_1^*\omega_{\mathcal{X}})$ is a Lie algebroid over $T \mathcal{M}(\mathcal{X},\alpha)$. This implies that $\left(R^1 (\mu_1)_* \mathcal{H}om(\mathscr{F} \otimes \nu_1^*\mathcal{L},\mathscr{F}\otimes \nu_1^*\omega_{\mathcal{X}}) \right)^*$ has a Poisson structure. We only have to prove that
	\begin{align*}
	R^1 (\mu_1)_* \mathcal{E}nd(\mathscr{F} \otimes \nu_1^*\mathcal{L},\mathscr{F}\otimes \nu_1^* \omega_{\mathcal{X}})^*\cong \mathcal{M}^0_H(\mathcal{X},\mathcal{L},\alpha),
	\end{align*}
	which would imply the result.
	
	By Grothendieck duality (see \S 3.3), we see that
	\begin{align*}
	{\rm Ext}^1(\mathcal{F}\otimes \mathcal{L},\mathcal{F}  \otimes \omega_{\mathcal{X}})^* \cong H^0(\mathcal{X},\mathcal{E}nd(\mathcal{F}) \otimes \mathcal{L}).
	\end{align*}
	Equivalently,
	\begin{align*}
	R^1 (\mu_1)_* \mathcal{H}om(\mathscr{F} \otimes \nu_1^*\mathcal{L},\mathscr{F}\otimes \nu_1^*\omega_{\mathcal{X}}) \cong (\mu_1)_* \mathcal{H}om(\mathscr{F} ,\mathscr{F}\otimes \nu_1^*\mathcal{L}).
	\end{align*}
	Therefore, $R^1 (\mu_1)_* \mathcal{H}om(\mathscr{F} \otimes \nu_1^*\mathcal{L},\mathscr{F}\otimes \nu_1^* \omega_{\mathcal{X}})^*$ is isomorphic to $\mathcal{M}^0_H(\mathcal{X},\mathcal{L},\alpha)$ and this finishes the proof of the theorem.
\end{proof}

\begin{rem}\label{502}
	Note that if the morphism $\mathcal{H}om(\mathcal{F}\otimes \mathcal{L},\mathcal{F}  \otimes \omega_{\mathcal{X}} )\rightarrow \mathcal{E}nd(\mathcal{F})$ is surjective, and $\mathfrak{n}$ is zero, then the term $\mathscr{A}d$ in (\ref{eq503}) is zero. This means that (\ref{eq503}) is not an Atiyah sequence, and the approach of Theorem \ref{501} fails in this case. Therefore, Theorem \ref{501} does not apply to the case $\mathcal{L}=\omega_{\mathcal{X}}$. However, this special case is treated exactly in Theorem \ref{404}.
\end{rem}

\section{Relation With the Work of Logares-Martens}
In this section, we compare the results of Sections 4 and 5 to the main result of Logares and Martens \cite{LoMa}. In the special case of a root stack and a parabolic structure when all parabolic weights are rational, we show that the moduli space of stable orbifold $G$-Higgs bundles has a Poisson structure. 

Let $X$ be a smooth projective curve. Let $\bar{D}=p_1 + \dots + p_k$ be a reduced effective divisor on $X$, and let $\bar{r}=(r_1,\dots,r_k)$ be a $k$-tuple of positive integers. We denote by $X_{\bar{D},\bar{r}}$ the corresponding root stack and there is a natural map $\pi: X_{\bar{D},\bar{r}} \rightarrow X$. Denote $\mathbb{D}:=\pi^{-1}(\bar{D})$. By the discussion in \S 2.5, we know that there is a correspondence between parabolic bundles on $(X,\bar{D})$ and bundles on $\mathcal{X}:=X_{\bar{D},\bar{r}}$. Moreover, we have the following equivalence in the language of tensor categories from \cite{Bor}:
\begin{prop}{{\cite[Th\'{e}or\`{e}me 4, 5]{Bor}}}\label{601}
	There is an equivalence of tensor categories between bundles $\mathcal{F}$ with parabolic structure $\alpha$ on $\mathcal{X}$ and parabolic bundles $F$ with the same parabolic structure $\alpha$ on $(X,\bar{D})$. In particular, this equivalence preserves the degree, that is, $\pardeg(F)=\deg(\mathcal{F})$.
\end{prop}

As discussed in \S 2.5, this equivalence implies the correspondence in stability as in \cite[Remarque 10]{Bor}. More precisely, we have $\mathcal{M}(\mathcal{X},\alpha)\cong \mathcal{M}^{par}(X,\alpha)$. It is natural to extend this correspondence to Higgs bundles \cite{BMW2, BMW,NaSt}, thus giving $\mathcal{M}_H(\mathcal{X},\alpha)\cong \mathcal{M}_H^{spar}(X,\alpha)$. Now let $\mathcal{L}$ be a line bundle on $\mathcal{X}$ and let $\pi:\mathcal{X}\to X$ be the map from the root stack $\mathcal{X}$ to its coarse moduli space $X$. Denote by $L$ the corresponding parabolic line bundle of $\mathcal{L}$ on $(X,\bar{D})$. There is a one-to-one correspondence between $\mathcal{L}$-twisted stable Higgs bundles on $\mathcal{X}$ and $L$-twisted stable parabolic bundles on $(X,\bar{D})$ (see \cite[\S 5]{KSZ2}).

In conclusion, we have the following proposition.
\begin{prop}\label{602}
	Let $\mathcal{X}=X_{\bar{D},\bar{r}}$ be a root stack and denote by $X$ the coarse moduli space of $\mathcal{X}$. The following statements hold:
	\begin{enumerate}
		\item There is an isomorphism $\mathcal{M}_H(\mathcal{X},\alpha)\cong \mathcal{M}^{spar}_H\left(X,\alpha\right)$, where $\mathcal{M}^{spar}_H\left(X,\alpha\right)$ is the moduli space of strongly parabolic Higgs bundles over $X$ with parabolic structure $\alpha$.
		\item There is an isomorphism $\mathcal{M}_H(\mathcal{X},\omega_{\mathcal{X}}(\mathbb{D}),\alpha)\cong \mathcal{M}^{par}_H\left(X,\alpha\right)$, where $\omega_{\mathcal{X}}(\mathbb{D})$ is the canonical line bundle on ${\mathcal{X}}$ over the divisor $\mathbb{D}$.
		\item Let $\mathcal{L}$ be an invertible sheaf on $\mathcal{X}$ and let $\pi:\mathcal{X}\to X$ be the map from the root stack to its coarse moduli space. Denote by $L$ the corresponding parabolic line bundle to $\mathcal{L}$ on $(X,\bar{D})$. Then there is an isomorphism $\mathcal{M}_H(\mathcal{X},\mathcal{L},\alpha)\cong \mathcal{M}^{par}_H\left(X,L,\alpha\right)$.
	\end{enumerate}
\end{prop}

Note that an orbifold is always a root stack. Thus, we have the following corollary to Theorem \ref{405}:
\begin{cor}\label{603}
	The moduli space of stable orbifold $G$-Higgs bundles has a Poisson structure.
\end{cor}

In \cite[\S 3.2.2]{LoMa}, the following short exact sequence is used to construct a Poisson structure on $\mathcal{M}^{par}_H(X,\alpha)$
\begin{equation}\tag{6.1}\label{eq601}
0\to \mathcal{SP}ar\mathcal{E}nd(F)\to \mathcal{P}ar\mathcal{E}nd(F)\to \prod_{p\in D} \mathfrak{l}_p\otimes \mathcal{O}_p\to 0,
\end{equation}
where $F$ is a parabolic bundle on $(X,\bar{D})$, $\mathcal{SP}ar\mathcal{E}nd(F)$ is the sheaf of strongly parabolic endomorphisms and $\mathcal{P}ar\mathcal{E}nd(F)$ is the sheaf of parabolic endomorphisms. Translating to the language of stacks, we have
\begin{align*}
\mathcal{H}om(\mathcal{F}(\mathbb{D}),\mathcal{F}  ) \cong \mathcal{SP}ar\mathcal{E}nd(F), \quad \mathcal{H}om(\mathcal{F},\mathcal{F}  ) \cong \mathcal{P}ar\mathcal{E}nd(F),
\end{align*}
where $\mathcal{F}$ is the bundle on root stacks corresponding to $F$. Therefore the sequence \eqref{eq601} is equivalent to
\begin{equation}\tag{6.2}\label{eq602}
0 \rightarrow  \mathcal{H}om(\mathcal{F}(\mathbb{D}),\mathcal{F}  ) \rightarrow \mathcal{H}om(\mathcal{F},\mathcal{F}) \rightarrow \prod_{p\in D} \mathfrak{l}_p\otimes \mathcal{O}_p \rightarrow 0.
\end{equation}
Note that when $\mathcal{L}=\omega_{\mathcal{X}}(\mathbb{D})$, the short exact sequence \eqref{eq501} becomes
\begin{align*}
0 \rightarrow  \mathcal{H}om(\mathcal{F}(\mathbb{D}),\mathcal{F}  ) \rightarrow \mathcal{H}om(\mathcal{F},\mathcal{F}) \rightarrow \mathfrak{n} \rightarrow 0.
\end{align*}
Clearly, the sequence \eqref{eq602} satisfies the conditions of Theorem \ref{501}. With respect to the above discussion, we can prove alternatively to \cite{LoMa} the following:

\begin{cor}\label{604}
	The moduli space $\mathcal{M}^{par}_H(X,\alpha)$ has a Poisson structure.
\end{cor}

\begin{proof}
	Given the data $(X,\bar{D},\alpha)$, we can construct a root stack $\mathcal{X}=X_{\bar{D},\bar{r}}$ (see \S 2.4 and \S 2.5). Denote by $\mathbb{D}$ the corresponding divisor on $\mathcal{X}$. There is a one-to-one correspondence between parabolic Higgs bundles with parabolic structure $\alpha$ and Higgs bundles with parabolic structure $\alpha$ on $\mathcal{X}$. Under this correspondence, the line bundle $\omega_X(\bar{D})$ on $X$ corresponds to $\omega_{\mathcal{X}}(\mathbb{D})$ on $\mathcal{X}$. By Proposition \ref{602}, this induces an isomorphism between $\mathcal{M}_H^{par}(X,\alpha)$ and $\mathcal{M}_H(\mathcal{X},\omega_{\mathcal{X}}(\mathbb{D}),\alpha)$. Therefore, it is enough to prove that the moduli space $\mathcal{M}_H(\mathcal{X},\omega_{\mathcal{X}}(\mathbb{D}),\alpha)$ has a natural Poisson structure. When $\mathcal{L}=\omega_{\mathcal{X}}(\mathbb{D})$, the condition in Theorem \ref{501} is automatically satisfied. This finishes the proof of the corollary.
\end{proof}

\vspace{2mm}
\textbf{Acknowledgments}.
The authors wish to warmly thank Philip Boalch and Florent Schaffhauser for helpful discussions and particularly useful comments. The authors also thank the anonymous referees for a careful reading of the manuscript and important remarks which led to various improvements. G. K. and H. S. are very grateful to Athanase Papadopoulos, Weixu Su and the ``Programme de Recherche conjoint (PRC) CNRS/NNSF of China 2018" for support for their visit to Fudan University, where part of this work was completed. L. Z. thanks the Institut de Recherche Math\'{e}matique Avanc\'{e}e of the Universit\'{e} de Strasbourg for its hospitality. G. K. is grateful to the Labex IRMIA of the Universit\'{e} de Strasbourg for support during the completion of this project. H. S. is supported by National Key R$\&$D Program of China No. 2022YFA1006600 and NSFC12101243.

\bigskip
\noindent\small{\textsc{Department of Mathematics, University of Patras}\\ Panepistimioupolis Patron, Patras 26504, Greece
	}\\
\emph{E-mail address}:  \texttt{gkydonakis@math.upatras.gr}

\bigskip
\noindent\small{\textsc{Department of Mathematics, South China University of Technology}\\
	381 Wushan Rd, Tianhe Qu, Guangzhou, Guangdong, China}\\
\emph{E-mail address}:  \texttt{hsun71275@scut.edu.cn}

\bigskip
\noindent\small{\textsc{Department of Mathematics, University of Maryland, College Park}\\
	4176 Campus Drive - William E. Kirwan Hall,
	College Park, MD 20742-4015, USA}\\
\emph{E-mail address}: \texttt{ltzhao@umd.edu}


\vspace{2mm}

\begin{thebibliography}{12}
	
	\bibitem{Atbook}
	M. F. Atiyah, \emph{The geometry and physics of knots}. Cambridge University Press, Cambridge (1990). 
	
	\bibitem{AtBo}
	M. F. Atiyah, R. Bott, The Yang--Mills equations over Riemann surfaces.  \emph{Philos. Trans. R. Soc. Lond.} \textbf{308} (1983), 523-615.
	
	\bibitem{Audin}
	M. Audin, \emph{Lectures on gauge theory and integrable systems}. Gauge theory and symplectic geometry (Montreal, PQ, 1995), 1–48, \emph{NATO Adv. Sci. Inst. Ser. C Math. Phys. Sci.} \textbf{488}, Kluwer Acad. Publ., Dordrecht (1997).
	
	\bibitem{BaBisNa}
	V. Balaji, I. Biswas, D. S. Nagaraj, Principal bundles over projective manifolds with parabolic structure over a divisor. \emph{Tohoku Math. J.} (2) \textbf{53} (2001), no. 3, 337-367.
	
	\bibitem{BiBo}
    O. Biquard, P. P. Boalch, Wild non-abelian Hodge theory on curves. \emph{Compos. Math.} \textbf{140} (2004), no. 1, 179-204.
	
	\bibitem{BiGaRi}
	O. Biquard, O. Garc\'{i}a-Prada, I. Mundet i Riera, Parabolic Higgs bundles and representations of the fundamental group of a punctured surface into a real group. \emph{Adv. Math.} \textbf{372} (2020), 107305.
	
	\bibitem{Biswas2}
	I. Biswas, Parabolic bundles as orbifold bundles.  \emph{Duke Math. J.} \textbf{88} (1997), no. 2, 305-325.
	
	\bibitem{BMW2}
	I. Biswas, S. Majumder, M. L. Wong, Parabolic Higgs bundles and $\Gamma $-Higgs Bundles. \emph{J. Aust. Math. Soc.} \textbf{95} (2013), 315-328.
	
	\bibitem{BMW}
	I. Biswas, S. Majumder, M. L. Wong, Root stacks, principal bundles and connections. \emph{Bull. Sci. Math.} \textbf{136} (2012), no. 4, 369-398.
	
	
	\bibitem{BisRam}
	I. Biswas, S. Ramanan, An infinitesimal study of the moduli of Hitchin pairs. \emph{J. London Math. Soc.} (2) \textbf{49} (1994), no. 2, 219-231.
	
	\bibitem{Bo3}
	P. P. Boalch, Geometry and braiding of Stokes data; fission and wild character varieties. \emph{Ann. of Math. (2)} \textbf{179} (2014), no. 1, 301-365.
	

	\bibitem{BoSurvey}
	P. P. Boalch, Poisson varieties from Riemann surfaces. \emph{Indag. Math.} \textbf{25} (2014), 872-900.
	
	\bibitem{Bo4}
	P. P. Boalch, Topology of the Stokes phenomenon. Novikov, Sergey (ed.) et al., Integrability, quantization, and geometry I. Integrable systems. Dedicated to the memory of Boris Dubrovin 1950-2019. Providence, RI: American Mathematical Society (AMS). \emph{Proc. Symp. Pure Math.} \textbf{103} (2021), Part 1, 55-100.
	
	\bibitem{Bor}
	N. Borne, Fibr\'{e}s paraboliques et champ des racines. \emph{Int. Math. Res. Not.} 2007, no. 16, 38pp.
	
	\bibitem{Bot}
	F. Bottacin, Symplectic geometry on moduli spaces of stable pairs. \emph{Ann. Sci. Ec. Norm. Sup.} (4) \textbf{28} (1995), no. 4, 391-433.
	
	\bibitem{Cad}
	C. Cadman, Using stacks to impose tangency conditions on curves. \emph{Am. J. Math.} \textbf{129} (2007), no. 2, 405-427.
	
	\bibitem{CasaWise}
	S. Casalaina-Martin, J. Wise, An introduction to moduli stacks, with a view towards Higgs bundles on algebraic curves. The geometry, topology and physics of moduli spaces of Higgs bundles, 199-399, \emph{Lect. Notes Ser. Inst. Math. Sci. Natl. Univ. Singap.} \textbf{36}, World Sci. Publ., Hackensack, NJ (2018).
	
	\bibitem{Cour}
	T. J. Courant, Dirac manifolds. \emph{Trans. Amer. Math. Soc.} \textbf{319} (1990), no. 2, 631-661.
	
	\bibitem{FuSt}
	M. Furuta, B. Steer, Seifert fibred homology 3-spheres and the Yang--Mills equations on Riemann surfaces with marked points. \emph{Adv. Math.} \textbf{96} (1992), no. 1, 38-102.
	
	\bibitem{Go}
	W. M. Goldman, The symplectic nature of fundamental groups of surfaces.  \emph{Adv. Math.} \textbf{54} (1984), 200-225.

    \bibitem{HarMor}
	J. Harris, I. Morrison, \emph{Moduli of curves}. Graduate Texts in Mathematics 187. Springer--Verlag (1988).
	
	\bibitem{Hit87}
	N. J. Hitchin, The self-duality equations on a Riemann surface. \emph{Proc. Lond. Math. Soc.} \textbf{55} (3) (1987), 59-126.
	

    \bibitem{Kaledin}
    D. Kaledin, Symplectic singularities from the Poisson point of view. \emph{J. Reine Angew. Math.} \textbf{600} (2006), 135-156.
	
	\bibitem{KSZ}
	G. Kydonakis, H. Sun, L. Zhao, Topological invariants of parabolic $G$-Higgs bundles. \emph{Math. Z.} \textbf{297}  (2021), no. 1-2, 585-632.
	
	\bibitem{KSZ2}
	G. Kydonakis, H. Sun, L. Zhao, The Beauville--Narasimhan--Ramanan correspondence for twisted Higgs $V$-bundles and components of parabolic $\text{Sp}(2n,\mathbb{R})$-Higgs moduli spaces. \emph{Trans. Amer. Math. Soc.} \textbf{374}  (2021), no. 6, 4023-4057.

    \bibitem{KSZ23}
	G. Kydonakis, H. Sun, L. Zhao, Logahoric Higgs Torsors for a Complex Reductive Group. \emph{Math. Ann.} (2023).
 
	\bibitem{LoMa}
	M. Logares, J. Martens, Moduli of parabolic Higgs bundles and Atiyah algebroids. \emph{J. Reine Angew. Math.} \textbf{649} (2010), 89-116.
	
	\bibitem{Mar}
	E. Markman, Spectral curves and integrable systems. \emph{Compos. Math.} \textbf{93} (1994), no. 3, 255-290.
	
	\bibitem{NaSt}
	B. Nasatyr, B. Steer, Orbifold Riemann surfaces and the Yang--Mills--Higgs equations.  \emph{Ann. Scuola Norm. Sup. Pisa Cl. Sci. (4)} \textbf{22} (1995), no. 4, 595-643.
	
	\bibitem{Nir2}
	F. Nironi, Grothendieck Duality for Deligne--Mumford stacks. arXiv: 0811.1955 (2008).
	
	
	\bibitem{Ol}
	M. Olsson, \emph{Algebraic spaces and stacks}. American Mathematical Society Colloquium Publications, 62. American Mathematical Society, Providence, RI, 2016. xi+298 pp.
	
	\bibitem{OlSt}
	M. Olsson, J. Starr, Quot functors for Deligne--Mumford stacks. Special issue in honor of Steven L. Kleiman. \emph{Comm. Algebra} \textbf{31} (2003), no. 8, 4069-4096.
	
	\bibitem{Sa}
	C. Sabbah, Harmonic metrics and connections with irregular singularities. \emph{Ann. Inst. Fourier} \textbf{49} (1999), no. 4, 1265-1291.
	
	\bibitem{Simp-noncompact}
	C. T. Simpson, Harmonic bundles on noncompact curves. \emph{J. Amer. Math. Soc.} \textbf{3} (1990), no. 3, 713-770.
	
	\bibitem{Simp2010}
	C. T. Simpson, Local systems on proper algebraic $V$-manifolds. \emph{Pure Appl. Math. Q.} \textbf{7} (2011), no. 4, Special Issue: In memory of Eckart Viehweg, 1675-1759.
	
	\bibitem{Simp2}
	C. T. Simpson, Moduli of representations of the fundamental group of a smooth projective variety. I. \emph{Inst. Hautes \'{E}tudes Sci. Publ. Math.}  \textbf{79} (1994), 47-129.
	
	\bibitem{Simp3}
	C. T. Simpson, Moduli of representations of the fundamental group of a smooth projective variety. II. \emph{Inst. Hautes \'{E}tudes Sci. Publ. Math.} \textbf{80} (1994), 5-79.
	
	\bibitem{Sun191}
	H. Sun, Moduli Problem of Hitchin Pairs over Deligne-Mumford Stacks. \emph{Proc. Amer. Math. Soc.} \textbf{150} (2022), no.1, 131-143.
	
	\bibitem{Sun201}
	H. Sun, Moduli Space of $\Lambda $-modules on Projective Deligne--Mumford Stacks. arXiv:2003.11674 (2020).
	
	\bibitem{Wal}
	J. Waldron, \emph{Lie Algebroids over Differentiable Stacks}.  Ph.D. Thesis, University of York (2014).
	
	\bibitem{Yoko1}
	K. Yokogawa, Compactification of moduli of parabolic sheaves and moduli of parabolic Higgs sheaves. \emph{J. Math. Kyoto Univ.} \textbf{33} (1993), 451-504.
	
	\bibitem{Yoko2}
	K. Yokogawa, Infinitesimal deformation of parabolic Higgs sheaves. \emph{Int. J. Math.} \textbf{6} (1995), 125-148.
	
\end{thebibliography}
\end{document}